\documentclass[11pt]{article}
\newcommand{\filename}{Evolution-Anis-torus-solu6.tex}
\usepackage{amsthm}
\usepackage{amsfonts}
\usepackage{amssymb}
\usepackage{amsmath}
\usepackage{comment}
\usepackage[]{graphicx}
\graphicspath{{./figures/}{../figures/}{../../figures/}}
\usepackage[colorlinks=true]{hyperref}
\usepackage{color}

\newcommand{\mycomment}[1]{}
\newtheorem{lemma}{LEMMA}[section]
\newtheorem{theorem}[lemma]{THEOREM}

\newtheorem{definition}[lemma]{DEFINITION}

\newtheorem{remark}[lemma]{REMARK}

\newcommand{\be}{\begin{equation}}
\newcommand{\ee}{\end{equation}}
\newcommand{\bes}{\begin{equation*}}
\newcommand{\ees}{\end{equation*}}
\newcommand{\bea}{\begin{eqnarray}}
\newcommand{\eea}{\end{eqnarray}}
\newcommand{\beas}{\begin{eqnarray*}}
\newcommand{\eeas}{\end{eqnarray*}}
 \newcommand{\nc}{\newcommand}
 \nc{\ha}{\frac{1}{2}}
 \nc{\tha}{\frac{3}{2}}

 \nc{\ov}{\overline}
 \nc{\pa}{\partial}
 \nc{\pO}{{\partial\T}}

\nc{\C}{{\mathbb{C}}}

\newcommand{\F}{{\cal F}}

 \nc{\N}{{\mathbb{N}}}

\newcommand{\R}{{\mathbb R}}
\newcommand{\T}{{\mathbb T}}

\newcommand{\Z}{\mathbb{Z}}

\newcommand{\sZ}{\sum_{\xi\in\Z^n}}

\renewcommand{\Re}{{\rm Re}}

\newcommand{\bs}{\boldsymbol}

\renewcommand{\div}{{\rm div}}

\makeatletter \@addtoreset{equation}{section}\makeatother
 \makeatletter
\renewcommand{\@oddhead}{\vbox{\hbox to\textwidth{\scriptsize %
 \hfill S.E.Mikhailov\hfill \arabic{page}\vspace{1ex}}
 \hrule
 }}
 \renewcommand{\@evenhead}{\vbox{\hbox to\textwidth{\scriptsize %
 \filename\hfill NS-torus
 }
\hrule
 }}
 \makeatother
 \allowdisplaybreaks
 \numberwithin{equation}{section}

\begin{document}

\title
{\bf
Spatially-Periodic Solutions for Evolution Anisotropic \\
Variable-Coefficient Navier-Stokes Equations: \\
I. Existence}

\author
{Sergey E. Mikhailov\footnote{
e-mail: {\sf sergey.mikhailov@brunel.ac.uk}, 
}\\
     Brunel University London,
     Department of Mathematics,\\
     Uxbridge, UB8 3PH, UK}

 \maketitle


\begin{abstract}\noindent
We consider evolution (non-stationary) space-periodic solutions to the $n$-dimensional non-linear Navier-Stokes equations of anisotropic fluids with the viscosity coefficient tensor variable in space and time and satisfying the relaxed ellipticity condition. 
Employing the Galerkin algorithm with the basis constituted by the eigenfunctions of the periodic Bessel-potential operator, we prove the existence of a global weak solution.\\


\noindent {\bf MSC classes}:	35A1, 35B10, 35K45, 35Q30, 76D05
\end{abstract}

\section{Introduction}
Analysis of Stokes and Navier-Stokes equations is an established and active field of research in the applied mathematical analysis, see, e.g., \cite{Constantin-Foias1988, Galdi2011, Lions1969, RRS2016, Seregin2015, Sohr2001, Temam1995, Temam2001} and references therein.
In \cite{KMW2020, KMW-DCDS2021, KMW-LP2021, KMW-transv2022, Mikhailov2022, Mikhailov2023} this field has been extended to the transmission and boundary-value problems for stationary Stokes and Navier-Stokes equations of anisotropic fluids, particularly, with relaxed ellipticity condition on the viscosity tensor. 

In this paper, we consider evolution (non-stationary) space-periodic solutions in $\R^n$, $n\ge 2$, to the  Navier-Stokes equations of anisotropic fluids with the viscosity coefficient tensor variable in space coordinates and time and satisfying the relaxed ellipticity condition. 
By the Galerkin algorithm with the basis constituted by the eigenfunctions of the periodic Bessel-potential operator, the solution existence is analysed in the spaces of Banach-valued functions mapping a finite time interval to  periodic Sobolev (Bessel-potential) spaces on $n$-dimensional flat torus. 

\subsection*{Anisotropic Stokes and Navier-Stokes PDE systems}

Let $n\ge 2$ be an integer, $\mathbf x\in\mathbb R^n$ denote the space coordinate vector, and $t\in\R$ be time.
Let
$\boldsymbol{\mathfrak L}$ denote the second-order differential operator represented in the component-wise divergence form as
\begin{align}
\label{L-oper}
&(\boldsymbol{\mathfrak L}{\mathbf u})_k:=
\partial _\alpha\big(a_{kj}^{\alpha \beta }E_{j\beta }({\mathbf u})\big),\ \ k=1,\ldots ,n,
\end{align}
where ${\mathbf u}\!=\!(u_1,\ldots ,u_n)^\top$, $E_{j\beta }({\mathbf u})\!:=\!\frac{1}{2}(\partial_j u_\beta +\partial _\beta u_j)$ are the entries of the symmetric part, ${\mathbb E}({\mathbf u})$, of the gradient,  $\nabla {\mathbf u}$,  in space coordinates,
and $a_{kj}^{\alpha \beta }(\mathbf x,t)$ are variable components of the tensor viscosity coefficient, cf. \cite{Duffy1978}, 
${\mathbb A}(\mathbf x,t)=\!\left\{{a_{kj}^{\alpha \beta }}(\mathbf x,t)\right\}_{1\leq i,j,\alpha ,\beta \leq n}$, depending on the space coordinate vector $\mathbf x$ and time $t$.
We also denoted $\partial_j=\dfrac{\partial}{\partial x_j}$, $\partial_t=\dfrac{\partial}{\partial t}$.
Here and further on, the Einstein  convention on summation in repeated indices from $1$ to $n$ is used unless stated otherwise.

The following symmetry conditions are assumed (see \cite[(3.1),(3.3)]{Oleinik1992}),
\begin{align}
\label{Stokes-sym}
a_{kj}^{\alpha \beta }(\mathbf x,t)=a_{\alpha j}^{k\beta }(\mathbf x,t)=a_{k\beta }^{\alpha j}(\mathbf x,t).
\end{align}

In addition, we require that tensor ${\mathbb A}$ satisfies the relaxed ellipticity condition  in terms of all {\it symmetric} matrices in ${\mathbb R}^{n\times n}$ with {\it zero matrix trace}, see \cite{KMW-DCDS2021}, \cite{KMW-LP2021}. Thus, we assume that there exists a constant $C_{\mathbb A} >0$ such that, 
\begin{align}
\label{mu}
&C_{\mathbb A}a_{kj}^{\alpha \beta }(\mathbf x,t)\zeta _{k\alpha }\zeta _{j\beta }\geq |\bs\zeta|^2\,,
\ \ 
\mbox{for a.e. } \mathbf x, t,\\
&\forall\ \bs\zeta =\{\zeta _{k\alpha }\}_{k,\alpha =1,\ldots ,n}\in {\mathbb R}^{n\times n}
\mbox{ such that }\, \bs\zeta=\bs\zeta^\top \mbox{ and }
\sum_{k=1}^n\zeta _{kk}=0,
\nonumber
\end{align}
where $|\bs\zeta |=|\bs\zeta |_F:=(\zeta _{k\alpha }\zeta _{k\alpha })^{1/2}$ is the Frobenius matrix norm and the superscript $^\top $ denotes the transpose of a  matrix.
Note that in the more common, strong ellipticity condition, inequality \eqref{mu} should be satisfied for all matrices (not only symmetric with zero trace), which makes it much more restrictive.

We assume that $a_{ij}^{\alpha \beta}\in L_\infty(\R^n\times [0,T])$, where $[0,T]$ is some finite time interval, and the tensor ${\mathbb A}$  is endowed with the norm
\begin{align}\label{TensNorm}
\|{\mathbb A}\|:=\|{\mathbb A}\|_{L_\infty (\R^n\times[0,T]),F}
:=\left| \left\{\|a_{ij}^{\alpha \beta}\|_{L_\infty (\T\times[0,T])}\right\}_{\alpha,\beta,i,j=1}^n\right|_F<\infty,
\end{align}
where $\left|\left\{b_{ij}^{\alpha \beta}\right\}_{\alpha,\beta,i,j=1}^n\right|_F:=\left(b_{ij}^{\alpha \beta}b_{ij}^{\alpha \beta}\right)^{1/2}$ is the Frobenius norm of a 4-th order tensor.

Symmetry conditions \eqref{Stokes-sym} lead to the following equivalent form of the operator $\boldsymbol{\mathfrak L}$
\begin{equation}
\label{Stokes-0}
(\boldsymbol{\mathfrak L}{\mathbf u})_k=\partial _\alpha\big(a_{kj}^{\alpha \beta }\partial _\beta u_j\big),\ \ k=1,\ldots ,n.
\end{equation}

Let ${\mathbf u}(\mathbf x,t)$ be an unknown vector velocity field, $p(\mathbf x,t)$ be an unknown (scalar) pressure field, and ${\mathbf f}(\mathbf x,t)$ be a given vector field 
$\R^n $, where $t\in\R$ is the time variable. 
Then the linear PDE system
\begin{equation*}
\partial_t\mathbf u - \boldsymbol{\mathfrak L}{\mathbf u}+\nabla p={\bf f},\ {\rm{div}}\ {\mathbf u}=0,
\end{equation*}
determines the {\it anisotropic evolution incompressible Stokes system}.

The nonlinear system
\begin{align*}
\partial_t\mathbf u -\boldsymbol{\mathfrak L}{\mathbf u}+\nabla p+({\mathbf u}\cdot \nabla ){\mathbf u}={\mathbf f}\,, \ \ {\rm{div}} \, {\mathbf u}=0
\end{align*}
is the {\it evolution anisotropic incompressible Navier-Stokes system}, where we use the notation $({\mathbf u}\cdot \nabla ):=u_j\partial_j$.

In the {\it isotropic case}, the tensor ${\mathbb A}$ reduces to
\begin{align*}
a_{kj}^{\alpha \beta}{(\mathbf x,t)}={\lambda{(\mathbf x,t)} \delta _{k\alpha }\delta _{j\beta }}+\mu{(\mathbf x,t)} \left(\delta_{\alpha j}\delta _{\beta k}+\delta_{\alpha \beta }\delta _{kj}\right),\ 1\leq k,j, \alpha ,\beta \leq n\,,
\end{align*}
where $\lambda,\mu\in L_\infty (\R^n\times[0,T])$, and
$c_\mu^{-1}\leq\mu(\mathbf x,t) \leq {c_\mu}$ for a.e. $ \mathbf x $ and $t$,
with some constant $c_\mu>0$ (cf., e.g., Appendix III, Part I, Section 1 in \cite{Temam2001}).
Then it is immediate that condition \eqref{mu} is fulfilled with $C_{\mathbb A}=c_\mu/2$ and thus our results apply also to the Stokes
and Navier-Stokes systems in the {\it isotropic case}.
Moreover, \eqref{L-oper} becomes
\begin{equation*}
\boldsymbol{\mathfrak L}{\mathbf u}
=(\lambda +\mu)\nabla{\rm div}\,\mathbf u
+\mu\Delta \mathbf u
+(\nabla\lambda){\rm div}\,\mathbf u
+2(\nabla \mu)\cdot{\mathbb E}({\mathbf u}).
\end{equation*}
 Assuming $\lambda=0$ and $\mu=1$ we arrive at the classical mathematical formulations of isotropic, constant-coefficient Stokes and Navier-Stokes systems in the familiar form
\begin{align*}
\partial_t\mathbf u -\Delta{\mathbf u}+\nabla p+{({\mathbf u}\cdot \nabla ){\mathbf u}}={\mathbf f}\,, \ \ {\rm{div}} \, {\mathbf u}=0. 
\end{align*}

\section{Periodic function spaces}\label{S2}

Let us introduce some function spaces on torus, i.e., periodic function spaces (see, e.g., 
\cite[p.26]{Agmon1965},  \cite{Agranovich2015}, \cite{McLean1991}, \cite[Chapter 3]{RT-book2010}, \cite[Section 1.7.1]{RRS2016}  
\cite[Chapter 2]{Temam1995}, 
for more details).

Let $n\ge 1$  be an integer and $\T$ be the $n$-dimensional flat torus that can be parametrized as the semi-open cube $\T=\T^n=[0,1)^n\subset\R^n$, cf.  \cite[p. 312]{Zygmund2002}.
In what follows, ${\mathcal D}(\T)=\mathcal C^\infty(\T)$ denotes the (test) space of infinitely smooth real or complex functions on the torus.
As usual, $\N$ denotes the set of natural numbers,  $\N_0$ the set of natural numbers augmented by 0, and $\mathbb{Z}$ the set of integers.

Let   $\bs\xi \in \mathbb{Z}^n$ denote the $n$-dimensional vector with integer components. 
We will further need also the set 
$$\dot\Z^n:=\Z^n\setminus\{\mathbf 0\}.$$
Extending the torus parametrisation to $\R^n$, it is often useful to identify $\T$ with the quotient space $\R^n\setminus \Z^n$. 
Then the space of functions $\mathcal C^\infty(\T)$ on the torus can be identified with the space of $\T$-periodic (1-periodic) functions 
$\mathcal C^\infty_\#=\mathcal C^\infty_\#(\R^n)$ that consists of functions $\phi\in \mathcal C^\infty(\R^n)$ such that
\begin{align}
\label{E3.1}
\phi(\mathbf x+\bs\xi)=\phi(\mathbf x)\quad \forall\,  \bs\xi \in \mathbb{Z}^n.
\end{align}
Similarly, the Lebesgue space on the torus $L_{p}(\T)$, $1\le p\le\infty$,  can be identified with the periodic Lebesgue space $L_{p\#}=L_{p\#}(\R^n)$ that consists of functions $\phi\in L_{p,\rm loc}(\R^n)$, which satisfy the periodicity condition \eqref{E3.1} for a.e. $\mathbf x$.

The space dual  to $\mathcal D(\T)$, i.e.,  the space of linear bounded functionals on $\mathcal D(\T)$,  called the space of torus distributions, is denoted by $\mathcal D'(\T)$ and can be identified with the space of periodic distributions $\mathcal D'_\#$ acting on $\mathcal C^\infty_\#$.

The toroidal/periodic Fourier transform 
mapping a  function $g\in \mathcal C_\#^\infty$ to a set of its Fourier coefficients $\hat g$ is defined as (see, e.g., \cite[Definition 3.1.8]{RT-book2010})  
\begin{align*}
 \hat g(\bs\xi)=[\F_{\T} g](\bs\xi):=\int_{\T}e^{-2\pi i \mathbf x\cdot\bs\xi}g(\mathbf x)d\mathbf x,\quad \bs\xi\in\Z^n,
\end{align*}
and can be generalised to  the Fourier transform acting on a distribution $g\in\mathcal D'_\#$.

For any $\bs\xi\in\Z^n$, let $|\bs\xi|:=(\sum_{j=1}^n \xi_j^2)^{1/2}$ be the Euclidean norm in $\Z^n$ and let us denote
 \begin{align*}
\varrho(\bs\xi):=2\pi(1+|\bs\xi|^2)^{1/2}.
\end{align*}
Evidently,
\begin{align}\label{eq:mik9}
\frac{1}{2}\varrho(\bs\xi)^2\le |2\pi\bs\xi|^2\le \varrho(\bs\xi)^2\quad\forall\,\bs\xi\in \dot\Z^n.
\end{align}

Similar to \cite[Definition 3.2.2]{RT-book2010}, for $s\in\R$ we define the {\em periodic/toroidal Sobolev (Bessel-potential) spaces} $H_\#^s:=H_\#^s(\R^n):=H^s(\T)$ that consist of the torus distributions $ g\in\mathcal D'(\T)$, for which the norm
\begin{align}\label{eq:mik10}
\| g\|_{H_\#^s}:=\| \varrho^s\widehat g\|_{\bs\ell_2(\Z^n)}:=\left(\sum_{\bs\xi\in\Z^n}\varrho(\bs\xi)^{2s}|\widehat g(\bs\xi)|^2\right)^{1/2}
\end{align}
is finite, i.e., the series in \eqref{eq:mik10} converges.
Here $\| \cdot\|_{\bs\ell_2(\Z^n)}$ is the standard norm in the space of square summable sequences with indices in $\Z^n$.
By \cite[Proposition 3.2.6]{RT-book2010}, $H_\#^s$ is the Hilbert space with 
the inner (scalar) product in $H_\#^{s}$  defined as
\begin{align}\label{E3.3i}
(g,f)_{H_\#^s}
:=\sum_{\bs\xi\in\Z^n}\varrho(\bs\xi)^{2s}\hat g(\bs\xi)\overline{\hat f(\bs\xi)},\quad \forall\, g,f\in H_\#^{s},\ s\in\R,
\end{align}
where the bar denotes complex conjugate.
Evidently, $H_\#^{0}=L_{2\#}$.

The dual product between $g\in H_\#^s$ and $f\in H_\#^{-s}$, $s\in\R$, is defined (cf. \cite[Definition 3.2.8]{RT-book2010}) as
\begin{align}\label{E3.3}
\langle g,f\rangle_{\T}
:=\sum_{\bs\xi\in\Z^n}\hat g(\bs\xi)\hat f(-\bs\xi).
\end{align}
If $s=0$, i.e., $g,f\in L_{2\#}$,  then \eqref{E3.3i} and \eqref{E3.3} reduce to
$$
\langle g,f\rangle_{\T}=\int_{\T}g(\mathbf x)f(\mathbf x)d\mathbf x=(g,\bar f)_{L_{2\#}}.
$$
For real function $g,f\in L_{2\#}$ we, of course, have $\langle g,f\rangle_{\T}=(g,f)_{L_{2\#}}$.

For any $s\in\R$, the space  $H_\#^{-s}$ is Banach adjoint (dual) to  $H_\#^{s}$, i.e., $H_\#^{-s}=(H_\#^{s})^*$.
Similar to, e.g., \cite[p.76]{McLean2000} one can show that 
\begin{align*}
\| g\|_{H_\#^s}=\sup_{f\in H_\#^{-s}, f\ne 0}\frac{|\langle g,f\rangle_{\T}|}{\|f\|_{H_\#^{-s}}}.
\end{align*}

For $g\in H_\#^s$, $s\in\R$, and  $m\in\N_0$, let us consider the partial sums 
$$g_m(\mathbf x)=\sum_{\bs\xi\in\Z^n, |\bs\xi|\le m}\hat g(\bs\xi)e^{2\pi i \mathbf x\cdot\bs\xi}.$$ 
 Evidently, $g_m\in \mathcal C_\#^\infty$, $\hat g_m(\bs\xi)=\hat g(\bs\xi)$ if $|\bs\xi|\le m$ and  $\hat g_m(\bs\xi)=0$ if $|\bs\xi|> m$.
This implies that $\|g-g_m\|_{H_\#^s}\to 0$ as $m\to\infty$ and hence we can write 
\begin{align}\label{eq:mik11}
g(\mathbf x)=\sum_{\bs\xi\in\Z^n}\hat g(\bs\xi)e^{2\pi i \mathbf x\cdot\bs\xi},
\end{align}
where the Fourier series converges in the sense of norm \eqref{eq:mik10}.
Moreover, since $g$ is an arbitrary distribution from $H_\#^s$, this also implies that the space ${\mathcal C}^\infty_\#$ is dense in $H_\#^s$ for any $s\in\R$ (cf. \cite[Exercise 3.2.9]{RT-book2010}).

There holds the compact embedding $H_\#^t\hookrightarrow H_\#^s$ if $t>s$,  embeddings $H_\#^s\subset \mathcal C_\#^m$ if $m\in\N_0$, $s>m+\frac{n}{2}$, and moreover, $\bigcap_{s\in\R}H_\#^s={\mathcal C}^\infty_\#$ (cf. \cite[Exercises 3.2.10, 3.2.10, and Corollary 3.2.11]{RT-book2010}). 

Note that for each $s$, the periodic norm \eqref{eq:mik10} is equivalent to the periodic norm that we used in \cite{Mikhailov2022, Mikhailov2023}, which is obtained from \eqref{eq:mik10} by replacing there $\varrho(\bs\xi)=2\pi(1+|\bs\xi|^2)^{1/2}$ with $\rho(\bs\xi)=(1+|\bs\xi|^2)^{1/2}$.
We employ here the norm \eqref{eq:mik10} to simplify some norm estimates further in the paper.
Note also that the periodic norms on $H_\#^s$ are equivalent to the corresponding standard (non-periodic) Bessel potential norms on $\T$ as a cubic domain, see, e.g., \cite[Section 13.8.1]{Agranovich2015}.

By \eqref{eq:mik10}, 
$\| g\|^2_{H_\#^s}=|\widehat g(\mathbf 0)|^2 +| g|^2_{H_\#^s},$ 
where
\begin{align*}
| g|_{H_\#^s}:=\| \varrho^s\widehat g\|_{\bs\ell_2(\dot\Z^n)}:=\left(\sum_{\bs\xi\in\dot\Z^n}\varrho(\bs\xi)^{2s}|\widehat g(\bs\xi)|^2\right)^{1/2}
\end{align*}
is the seminorm in $H_\#^s$.

For any $s\in\R$, let us also introduce  the space
\begin{align*}
\dot H_\#^s:=\{g\in H_\#^s: \langle g,1\rangle_{\T}=0\}.
\end{align*}
The definition implies that if $g\in \dot H_\#^s$, then $\widehat g(\mathbf 0)=0$ and 
\begin{align}\label{eq:mik12}
\| g\|_{\dot H_\#^s}=\| g\|_{H_\#^s}=| g|_{H_\#^s}=\| \varrho^s\widehat g\|_{\bs\ell_2(\dot\Z^n)}\ .
\end{align}
The space $\dot H_\#^s$ is the Hilbert space with inner product inherited from \eqref{E3.3i}, that is,
\begin{align}\label{E3.3idot}
(g_1,g_2)_{\dot H_\#^s}
:=\sum_{\bs\xi\in\dot\Z^n}\varrho(\bs\xi)^{2s}\hat g_1(\bs\xi)\overline{\hat g_2(\bs\xi)},\quad \forall\, g_1,g_2\in \dot H_\#^{s},\ s\in\R.
\end{align}
Due to the Riesz representation theorem, the dual product between $g_1\in \dot H_\#^s$ and $f_2\in (\dot H_\#^{s})^*$, $s\in\R$, can be represented as 
\begin{align*}
\langle g_1,f_2\rangle_{\T}
:=\sum_{\bs\xi\in\dot\Z^n}\hat g_1(\bs\xi)\hat f_2(-\bs\xi)
=(g_1,g_2)_{\dot H_\#^s}
=\sum_{\bs\xi\in\dot\Z^n}\varrho(\bs\xi)^{2s}\hat g_1(\bs\xi)\overline{\hat g_2(\bs\xi)}.
\end{align*}
where 
\begin{align*}
\hat f_2(\bs\xi)=\varrho(\bs\xi)^{2s}\overline{\hat g_2(-\bs\xi)},\quad 
\hat g_2(\bs\xi)=\varrho(\bs\xi)^{-2s}\overline{\hat f_2(-\bs\xi)},\quad
\bs\xi\in \dot\Z^n
\end{align*}
for some $g_2\in \dot H_\#^s$.
This implies that 
\begin{align}\label{E3.13}
 f_2(\mathbf x)=\left(\Lambda_\#^{2s}\,\overline{g_2}\right)(\mathbf x),
\end{align}
where $\Lambda_\#^{r}: H_\#^s\to H_\#^{s-r}$ is the continuous periodic Bessel-potential operator of the order $r\in\R$ defined as
\begin{align}\label{E3.14}
 \left(\Lambda_\#^{r}\,{g}\right)(\mathbf x)
:=\sum_{\bs\xi\in\Z^n}\varrho(\bs\xi)^{r}\hat g(\bs\xi)e^{2\pi i \mathbf x\cdot\bs\xi}\quad \forall\, g\in  H_\#^{s},\ s\in\R,
\end{align}
see, e.g., \cite[Section 13.8.1]{Agranovich2015}.
Note that \eqref{E3.14} implies
\begin{align*}
 \left(\Lambda_\#^{2}\,{g}\right)(\mathbf x)
=\sum_{\bs\xi\in\Z^n}(2\pi)^2(1+|\bs\xi|^2)\hat g(\bs\xi)e^{2\pi i \mathbf x\cdot\bs\xi}
=(2\pi)^2 g(\mathbf x) - \Delta^2 g(\mathbf x)
\quad \forall\, g\in  H_\#^{s},\ s\in\R.
\end{align*}

If $\hat g(\mathbf 0)=0$ then \eqref{E3.14} implies that $\widehat{\Lambda_\#^{r}\,g}(\mathbf 0)=0$, and 
thus the operator 
\begin{align}\label{E3.15}
\Lambda_\#^{r}: \dot H_\#^s\to \dot H_\#^{s-r}
\end{align}
is continuous as well.
Hence by \eqref{E3.13} we conclude that
$(\dot H_\#^{s})^*=\dot H_\#^{-s}$.

Denoting 
$
\dot {\mathcal C}^\infty_\#:=\{g\in {\mathcal C}^\infty_\#: \langle g,1\rangle_{\T}=0\}
$,
then $\bigcap_{s\in\R}\dot H_\#^s=\dot {\mathcal C}^\infty_\#$.

The corresponding spaces of $n$-component vector functions/distributions are denoted as $\mathbf L_{q\#}:=(L_{q\#})^n$, $\mathbf H_\#^s:=(H_\#^s)^n$, etc.

Note that  the norm $\|\nabla (\cdot )\|_{{\mathbf H}_\#^{s-1}}$
is an equivalent norm in $\dot H_\#^s$. 
Indeed, by \eqref{eq:mik11}
\begin{align*}
\nabla g(\mathbf x)=2\pi i\sum_{\bs\xi\in\dot\Z^n}\bs\xi e^{2\pi i \mathbf x\cdot\bs\xi}\hat g(\bs\xi),\quad
\widehat{\nabla g}(\bs\xi)=2\pi i\bs\xi \hat g(\bs\xi)\quad \forall\,g\in \dot H_\#^s,
\end{align*}
and 
then \eqref{eq:mik9} and \eqref{eq:mik12} imply
\begin{align}
\frac12|g|^2_{H_\#^s}&\le \| \nabla g\|^2_{{\mathbf H}_\#^{s-1}}
\le |g|^2_{H_\#^s} \quad \forall\,g\in H_\#^s,
\nonumber\\
\label{eq:mik13}
\frac12\|g\|^2_{H_\#^s}=\frac12\|g\|^2_{\dot H_\#^s}=\frac12|g|^2_{H_\#^s}&\le \| \nabla g\|^2_{{\mathbf H}_\#^{s-1}}
\le |g|^2_{H_\#^s}=\|g\|^2_{\dot H_\#^s}=\| g\|^2_{H_\#^s} \quad \forall\,g\in \dot H_\#^s.
\end{align}
The vector counterpart of \eqref{eq:mik13} takes form
\begin{align}\label{eq:mik14}
\frac12\| \mathbf v\|^2_{{\mathbf H}_\#^s}=\frac12\| \mathbf v\|^2_{\dot{\mathbf H}_\#^s}
\le \| \nabla \mathbf v\|^2_{(H_\#^{s-1})^{n\times n}}
\le \| \mathbf v\|^2_{\dot{\mathbf H}_\#^s}=\| \mathbf v\|^2_{{\mathbf H}_\#^s} \quad \forall\,\mathbf v\in \dot {\mathbf H}_\#^s.
\end{align}
Note that the second inequalities in \eqref{eq:mik13} and \eqref{eq:mik14} are valid also in the more general cases, i.e., for $\,g\in H_\#^s$ and $\mathbf v\in  {\mathbf H}_\#^s$, respectively.

We will further need the first Korn inequality
\begin{align}
\label{eq:mik15}
\|\nabla {\mathbf v}\|^2_{(L_{2\#})^{n\times n}}\leq 2\|\mathbb E ({\mathbf v})\|^2_{(L_{2\#})^{n\times n}}\quad\forall\, \mathbf v\in {\mathbf H}_\#^1 
\end{align}
that can be easily proved by adapting, e.g., the proof in \cite[Theorem 10.1]{McLean2000} to the periodic Sobolev space; cf. also \cite[Theorem 2.8]{Oleinik1992}.

Let us also define the Sobolev spaces of divergence-free functions and distributions,
\begin{align*}
\dot{\mathbf H}_{\#\sigma}^{s}
&:=\left\{{\bf w}\in\dot{\mathbf H}_\#^{s}:{\div}\, {\bf w}=0\right\},\quad s\in\R,
\end{align*}
endowed with the same norm \eqref{eq:mik10}.
Similarly, ${\mathbf C}^\infty_{\#\sigma}$ and $\mathbf L_{q\#\sigma}$ denote the subspaces of divergence-free vector-functions from 
${\mathbf C}^\infty_{\#}$ and $\mathbf L_{q\#}$, respectively, etc.

The space $\dot{\mathbf H}_{\#\sigma}^s$ is the Hilbert space with inner product inherited from \eqref{E3.3i} and \eqref{E3.3idot}, that is,
\begin{align*}
(\mathbf g_1,\mathbf g_2)_{\dot{\mathbf H}_{\#\sigma}^s}
:=\sum_{\bs\xi\in\dot\Z^n}\varrho(\bs\xi)^{2s}\widehat{\mathbf g}_1(\bs\xi)\overline{\widehat{\mathbf g}_2(\bs\xi)},\quad \forall\, \mathbf g_1,\mathbf g_2\in \dot{\mathbf H}_{\#\sigma}^{s},\ s\in\R.
\end{align*}
Due to the Riesz representation theorem, the dual product between $\mathbf g_1\in \dot{\mathbf H}_{\#\sigma}^s$ and ${\mathbf f}_2\in (\dot{\mathbf H}_{\#\sigma}^{s})^*$, $s\in\R$, can be represented as 
\begin{align*}
\langle \mathbf g_1,{\mathbf f}_2\rangle_{\T}
:=\sum_{\bs\xi\in\dot\Z^n}\widehat{\mathbf g}_1(\bs\xi)\widehat{\mathbf f}_2(-\bs\xi)
=(\mathbf g_1,\mathbf g_2)_{\dot{\mathbf H}_{\#\sigma}^s}
=\sum_{\bs\xi\in\dot\Z^n}\varrho(\bs\xi)^{2s}\widehat{\mathbf g}_1(\bs\xi)\overline{\widehat{\mathbf g}_2(\bs\xi)}.
\end{align*}
where 
\begin{align*}
\widehat{\mathbf f}_2(\bs\xi)=\varrho(\bs\xi)^{2s}\overline{\widehat{\mathbf g}_2(-\bs\xi)},\quad \bs\xi\in \dot\Z^n
\end{align*}
for some $\mathbf g_2\in \dot{\mathbf H}_{\#\sigma}^s$.
This implies that 
\begin{align}\label{E3.13-sigma}
 {\mathbf f}_2(\mathbf x)=\left(\Lambda_\#^{2s}\,\overline{\mathbf g_2}\right)(\mathbf x),
\end{align}
where
the operator 
\begin{align}\label{3.23}
\Lambda_\#^{r}: \dot{\mathbf H}_{\#\sigma}^s\to \dot{\mathbf H}_{\#\sigma}^{s-r}
\end{align}
defined as in \eqref{E3.14} is continuous.
Hence by \eqref{E3.13-sigma} we conclude that
$$
(\dot{\mathbf H}_{\#\sigma}^{s})^*=\dot{\mathbf H}_{\#\sigma}^{-s}.
$$

Let us also introduce the space 
\begin{align*}
\dot{\mathbf H}_{\# g}^{s}
&:=\left\{{\bf w}=\nabla q,\ q\in\dot{H}_\#^{s+1}\right\},\quad s\in\R,
\end{align*}
endowed with the norm \eqref{eq:mik10}.

Let $s\in\R$, $\mathbf w\in \dot{\mathbf H}_{\# g}^{s}$ and $\mathbf v\in \dot{\mathbf H}_{\# \sigma}^{s}$
By \eqref{E3.3i}, for their inner product in  $\dot{\mathbf H}_{\#}^{s}$ we obtain
\begin{multline*}
(\mathbf w,\mathbf v)_{H_\#^s}
:=\sum_{\bs\xi\in\Z^n}\varrho(\bs\xi)^{2s}\widehat{\mathbf w}(\bs\xi)\cdot\overline{\widehat{\mathbf v}(\bs\xi)}
=\sum_{\bs\xi\in\Z^n}\varrho(\bs\xi)^{2s}2\pi i\bs\xi\widehat{q}(\bs\xi)\cdot\overline{\widehat{\mathbf v}(\bs\xi)}
\\
=-\sum_{\bs\xi\in\Z^n}\varrho(\bs\xi)^{2s}\widehat{q}(\bs\xi)\overline{2\pi i\bs\xi\cdot\widehat{\mathbf v}(\bs\xi)}
=-\sum_{\bs\xi\in\Z^n}\varrho(\bs\xi)^{2s}\widehat{q}(\bs\xi)\overline{\widehat{\div\,\mathbf v}(\bs\xi)}=0.
\end{multline*}
Hence $\dot{\mathbf H}_{\# g}^{s}$ and $\dot{\mathbf H}_{\# \sigma}^{s}$ are orthogonal subspaces of $\dot{\mathbf H}_{\#}^{s}$ in the sense of inner product.

On the other hand, if $s\in\R$, $\mathbf w\in \dot{\mathbf H}_{\# g}^{s}$ and $\mathbf v\in \dot{\mathbf H}_{\# \sigma}^{-s}$, then
for their dual product  we obtain
\begin{align*}
\langle\mathbf w,\mathbf v\rangle=\langle\nabla q,\mathbf v\rangle=-\langle q,\div\,\mathbf v\rangle=0.
\end{align*}
Hence the spaces $\dot{\mathbf H}_{\# g}^{s}$ and $\dot{\mathbf H}_{\# \sigma}^{-s}$ are orthogonal in the sense of dual product.

For $s\in\R$ and $\mathbf F\in \dot{\mathbf H}_{\#}^{s}$, let us introduce the operators $\mathbb P_g$ and $\mathbb P_\sigma$ as follows,
\begin{align*}
&(\mathbb P_g\, \mathbf F)(\mathbf x):=
\sum_{\bs\xi\in\dot\Z^n}
\bs\xi\frac{\bs\xi\cdot\widehat{\mathbf F}(\bs\xi)}{|\bs\xi|^2}e^{2\pi i \mathbf x\cdot\bs\xi},
\\
&(\mathbb P_\sigma\, \mathbf F)(\mathbf x):=
\sum_{\bs\xi\in\dot\Z^n}\left(\widehat{\mathbf F}(\bs\xi) 
-\bs\xi\frac{\bs\xi\cdot\widehat{\mathbf F}(\bs\xi)}{|\bs\xi|^2}\right)e^{2\pi i \mathbf x\cdot\bs\xi}.
\end{align*}
Note that 
\begin{align}\label{E2.30}
\mathbf F(\mathbf x)=(\mathbb P_\sigma\, \mathbf F)(\mathbf x)+(\mathbb P_g\, \mathbf F)(\mathbf x)\quad
\forall\, \mathbf F\in \dot{\mathbf H}_{\#}^{s},\ s\in\R.
\end{align}
Evidently
\begin{align*}
&(\mathbb P_g\, \mathbf F)(\mathbf x)=\nabla q(\mathbf x),
\text{ where }
q(\mathbf x)=\sum_{\bs\xi\in\dot\Z^n}
\frac{\bs\xi\cdot\widehat{\mathbf F}(\bs\xi)}{2\pi i|\bs\xi|^2}e^{2\pi i \mathbf x\cdot\bs\xi},
\end{align*}
hence $q\in \dot{H}_{\#}^{s+1}$. 

One can check that $\mathbb P_g(\mathbb P_g\, \mathbf F)=\mathbb P_g\, \mathbf F$ and thus 
$
\mathbb P_g:\dot{\mathbf H}_\#^{s}\to \dot{\mathbf H}_{\#g}^{s}
$
is a bounded projector.
On the other hand,
$\div\,\mathbb P_\sigma\, \mathbf F=0$, $\mathbb P_\sigma(\mathbb P_\sigma\, \mathbf F)=\mathbb P_\sigma\, \mathbf F$ and hence
$
\mathbb P_\sigma:\dot{\mathbf H}_\#^{s}\to \dot{\mathbf H}_{\#\sigma}^{s}
$
is also a bounded projector.
Since $\dot{\mathbf H}_{\# g}^{s}$ and $\dot{\mathbf H}_{\# \sigma}^{s}$ are orthogonal subspaces of $\dot{\mathbf H}_{\#}^{s}$, the projectors $\mathbb P_g$ and $\mathbb P_\sigma$ are orthogonal in $\dot{\mathbf H}^{s}$.
The projector $\mathbb P_\sigma$ is called the Leray projector (see, e.g., \cite[Section 2.1]{RRS2016}).

Decomposition \eqref{E2.30} implies the representation 
$\dot{\mathbf H}_{\#}^{s}= \dot{\mathbf H}_{\# g}^{s}\oplus \dot{\mathbf H}_{\#\sigma}^{s}$ called  the Helmholtz-Weyl decomposition.
Note that the orthogonality of $\dot{\mathbf H}_{\# g}^{s}$ and $\dot{\mathbf H}_{\# \sigma}^{s}$ implies that for any $\mathbf F\in \dot{\mathbf H}_{\#}^{s}$, the representation $\mathbf F=\mathbf F_g+ \mathbf F_\sigma$, where $\mathbf F_g\in \dot{\mathbf H}_{\# g}^{s}$ and $\mathbf F_\sigma\in \dot{\mathbf H}_{\#\sigma}^{s}$, is unique and hence is given by  \eqref{E2.30}.

Summarising the obtained results, we arrive at the following assertion (cf., e.g., \cite[Theorem 2.6]{RRS2016}, where a similar statement is proved for $s=0$ and $n=3$). 
\begin{theorem} 
Let $s\in\R$ and $n\ge 2$. 

(a) The space $\dot{\mathbf H}_{\#}^{s}$ has the Helmholtz-Weyl decomposition,
$
\dot{\mathbf H}_{\#}^{s}=\dot{\mathbf H}_{\# g}^{s}\oplus\dot{\mathbf H}_{\# \sigma}^{s},
$
that is, any $\mathbf F\in \dot{\mathbf H}_{\#}^{s}$ can be uniquely represented as 
$\mathbf F=\mathbf F_g+ \mathbf F_\sigma$, where $\mathbf F_g\in \dot{\mathbf H}_{\# g}^{s}$ and 
$\mathbf F_\sigma\in \dot{\mathbf H}_{\#\sigma}^{s}$.

(b) The spaces $\dot{\mathbf H}_{\# g}^{s}$ and $\dot{\mathbf H}_{\# \sigma}^{s}$ are orthogonal subspaces of $\dot{\mathbf H}_{\#}^{s}$ in the sense of inner product, i.e., $(\mathbf w,\mathbf v)_{H_\#^s}=0$ for any $\mathbf w\in \dot{\mathbf H}_{\# g}^{s}$ and $\mathbf v\in \dot{\mathbf H}_{\# \sigma}^{-s}$.

(c) The spaces $\dot{\mathbf H}_{\# g}^{s}$ and $\dot{\mathbf H}_{\# \sigma}^{-s}$ are orthogonal in the sense of dual product, i.e., $\langle\mathbf w,\mathbf v\rangle=0$ for any $\mathbf w\in \dot{\mathbf H}_{\# g}^{s}$ and 
$\mathbf v\in \dot{\mathbf H}_{\# \sigma}^{-s}$.

(d) There exist the bounded orthogonal projector operators $\mathbb P_g:\dot{\mathbf H}_\#^{s}\to \dot{\mathbf H}_{\#g}^{s}$ and $\mathbb P_\sigma:\dot{\mathbf H}_\#^{s}\to \dot{\mathbf H}_{\#\sigma}^{s}$ (the Leray projector), while $\mathbf F=\mathbb P_g\mathbf F+\mathbb P_\sigma\mathbf F$ for any $\mathbf F\in \dot{\mathbf H}_\#^{s}$.
\end{theorem}

For the evolution problems we will systematically use the spaces  $L_q(0,T;H^s_\#)$, $s\in\R$, $1\le q\le\infty$, $0<T<\infty$, which consists of functions that map $t\in(0,T)$ to a function or distributions from $H^s_\#$.
For $1\le q<\infty$, the space $L_q(0,T;H^s_\#)$ is endowed with the norm
\begin{align*}
 \|h\|_{L_q(0,T;H^s_\#)}=\left(\int_0^T\|h(\cdot,t)\|^q_{H^s_\#} d t\right)^{1/q}
=\left(\int_0^T\left[\sZ\varrho(\bs\xi)^{2s}|\widehat h(\bs\xi,t)|^2\right]^{q/2} d t\right)^{1/q}<\infty,
\end{align*}
and for $q=\infty$ with the norm
\begin{align*}
 \|h\|_{L_\infty(0,T;H^s_\#)}={\rm ess\,sup}_{t\in(0,T)}\|h(\cdot,t)\|_{H^s_\#}
={\rm ess\,sup}_{t\in(0,T)}\left[\sZ\varrho(\bs\xi)^{2s}|\widehat h(\bs\xi,t)|^2\right]^{1/2}<\infty.
\end{align*}

For a function (or distribution) $h(\mathbf x,t)$, we will use the notation $h'(\mathbf x,t):=\partial_t h(\mathbf x,t):=\dfrac{\partial}{\partial t}h(\mathbf x,t)$ for the partial derivative in the time variable $t$. 
Let $X$ and $Y$ be some Hilbert spaces. We will further need the space 
\begin{align*}
W^1(0,T;X, Y):=\{u\in L_2(0,T; X) : u'\in L_2(0,T; Y)\}
\end{align*}
endowed with the norm
\begin{align*}
\|u\|_{W^1(0,T;X, Y)}=(\|u\|_{L_2(0,T; X)}^2 +\|u'\|_{L_2(0,T; Y)}^2)^{1/2}.
\end{align*}
Spaces of such type are considered in \cite[Chapter 1, Section 2.2]{Lions-Magenes1}. 
We will particularly need the spaces $W^1(0,T;H^s_\#, H^{s'}_\#)$  and their vector counterparts.

Unless stated otherwise, we will assume in this paper that all vector and scalar variables are real-valued (however with complex-valued Fourier coefficients).

\section{Weak formulation of the evolution spatially-periodic anisotropic Navier-Stokes problem}

Let us consider the following Navier-Stokes problem
for the real-valued unknowns $({\mathbf u},p )$, 
\begin{align}
\label{NS-problem-div0}
\partial_t\mathbf u -\bs{\mathfrak L}{\mathbf u}+\nabla p+({\mathbf u}\cdot \nabla ){\mathbf u}&=\mathbf{f}
\quad \mbox{in } \T\times(0,T), 
\\
\label{NS-problem-div0-div}
{\rm{div}}\, {\mathbf u}&=0\quad \mbox{in } \T\times(0,T),
\\
\label{NS-problem-div0-IC}
\mathbf u(\cdot,0)&=\mathbf u^0\ \mbox{in } \T ,
\end{align}
with given data
${\mathbf f}\in L_2(0,T;\dot{\mathbf H}_\#^{-1})$,  $\mathbf u^0\in \dot{\mathbf H}_{\#\sigma}^{0}$. 
Note that the time-trace $\mathbf u(\cdot,0)$ for $\mathbf u$ solving the weak form of \eqref{NS-problem-div0}--\eqref{NS-problem-div0-div} is well defined, see Definition \ref {D6.1} and Remark \ref{R4.3}.

Let us introduce the bilinear form
\begin{align}
\label{NS-a-vsigma}
a_{\T}({\mathbf u},{\mathbf v})=
a_{\T}(t;{\mathbf u},{\mathbf v})
:=
\left\langle a_{ij}^{\alpha \beta }(\cdot,t)E_{j\beta }({\mathbf u}),E_{i\alpha }({\mathbf v})\right\rangle _{\T}
&\,\
\forall \ {\mathbf u}, {\mathbf v}\in \dot{\mathbf H}_{\#}^1.
\end{align}
By the boundedness condition \eqref{TensNorm} and inequality \eqref{eq:mik14} we have
\begin{align}
\label{NS-a-1-v2-S-0v}
|a_{\T}(t;{\mathbf u},{\mathbf v})|
\le \|\mathbb A\| \|{\mathbb E}({\mathbf u})\|_{L_{2\#}^{n\times n}}\|{\mathbb E}({\bf v})\|_{L_{2\#}^{n\times n}}
&\le \|\mathbb A\| \|\nabla{\mathbf u}\|_{L_{2\#}^{n\times n}}\|\nabla{\bf v}\|_{L_{2\#}^{n\times n}}
\nonumber\\
&\le \|\mathbb A\| \|{\mathbf u}\|_{\dot{\mathbf H}_{\#}^1}\|{\bf v}\|_{\dot{\mathbf H}_{\#}^1}
\quad
\forall \ {\mathbf u}, {\mathbf v}\in \dot{\mathbf H}_{\#}^1.
\end{align}
If the relaxed ellipticity condition \eqref{mu} holds, taking into account the relation $\sum_{i=1}^nE_{ii}({\bf w})=\div {\bf w}=0$ for  ${\bf w}\in 
{\dot{\mathbf H}_{\#\sigma}^1}$, 
equivalence of the norm $\|\nabla (\cdot )\|_{L_{2\#}^{n\times n}}$
to the norm $\|\cdot \|_{\dot{\mathbf H}_{\#\sigma}^1}$ in $\dot{\mathbf H}_{\#\sigma}^1$, see \eqref{eq:mik14}, and the first Korn inequality \eqref{eq:mik15},
we obtain
\begin{align}
\label{NS-a-1-v2-S-0}
a_{\T}(t;{\bf w},{\mathbf w})=\left\langle a_{ij}^{\alpha \beta}(\cdot,t)E_{j\beta }({\bf w}),E_{i\alpha }({\bf w})\right\rangle _{\T}
&\geq C_{\mathbb A}^{-1}\|{\mathbb E}({\bf w})\|_{L_{2\#}^{n\times n}}^2
\nonumber\\
&\geq\frac{1}{2}C_{\mathbb A}^{-1}\|\nabla {\bf w}\|_{L_{2\#}^{n\times n}}^2
\geq \frac14 C_{\mathbb A}^{-1}\|{\bf w}\|_{\dot{\mathbf H}_{\#\sigma}^1}^2
\quad\forall \, {\bf w}\in \dot{\mathbf H}_{\#\sigma}^1.
\end{align}
Then  \eqref{NS-a-1-v2-S-0v} and \eqref{NS-a-1-v2-S-0} give
\begin{align}
\label{NS-a-1-v2-S-}
\frac14 C_{\mathbb A}^{-1}\|{\bf w}\|_{\dot{\mathbf H}_{\#\sigma}^1}^2
\le a_{\T}(t;{\bf w},{\mathbf w})
\le\|\mathbb A\| \|{\bf w}\|_{\dot{\mathbf H}_{\#\sigma}^1}^2\quad
\forall \, {\bf w}\in \dot{\mathbf H}_{\#\sigma}^1.
\end{align}
This inequality implies that $\sqrt{a_{\T}(t;{\bf w},{\mathbf w})}$ is an equivalent norm 
in $\dot{\mathbf H}_{\#\sigma}^1$ for almost any $t$ and, moreover, 
\begin{align}\label{eqnorm}
|\!|\!|\mathbf w|\!|\!|_{L_2(0,T;\dot{\mathbf H}_{\#}^1)}:=\displaystyle \left(\int_0^Ta_{\T}(t;{\bf w}(\cdot,t),{\mathbf w}(\cdot,t))dt\right)^{1/2}
\end{align} 
is an equivalent norm 
in $L_2(0,T;\dot{\mathbf H}_{\#\sigma}^{1})$.

We use the following definition of weak solution, 
that for $n\in\{2,3,4\}$ reduces to the weak formulations employed, e.g., in 
 \cite[Chapter 1, Problem 6.2]{Lions1969},
\cite[Definition 8.5]{Constantin-Foias1988},
\cite[Problem 2.1]{Temam1995}, \cite[Chapter 3, Problem 3.1]{Temam2001}.
However the definition that we use is applicable also to higher dimensions (and allows for those dimensions more general test functions than in \cite[Chapter 1, Problem 6.2]{Lions1969}). 
\begin{definition}\label{D6.1} Let $n\ge 2$, $T>0$, ${\mathbf f}\in L_2(0,T;\dot{\mathbf H}_\#^{-1})$ and  
$\mathbf u^0\in \dot{\mathbf H}_{\#\sigma}^{0}$.
A function ${\mathbf u}\in  L_{\infty}(0,T;\dot{\mathbf H}_{\#\sigma}^{0})\cap L_2(0,T;\dot{\mathbf H}_{\#\sigma}^{1})$ is called a weak solution of the evolution space-periodic anisotropic Navier-Stokes initial value problem \eqref{NS-problem-div0}--\eqref{NS-problem-div0-IC} if it solves  the initial-variational problem
\begin{align}
\label{NS-eq:mik51a}
&\left\langle {\mathbf u}'(\cdot,t)+\mathbb P_\sigma[({\mathbf u}(\cdot,t)\cdot \nabla ){\mathbf u}(\cdot,t)],{\mathbf w}\right\rangle _{\T }
+a_{\T}({\mathbf u}(\cdot,t),{\mathbf w})
=\langle\mathbf{f}(\cdot,t), \mathbf w\rangle_{\T},\ 
\text{for a.e. } t\in(0,T),\
\forall\,\mathbf w\in \dot{\mathbf H}_{\#\sigma}^{1},
\\ 
\label{NS-eq:mik51in}
&\langle {\mathbf u}(\cdot,0), \mathbf w\rangle_{\T}=\langle\mathbf u^0, \mathbf w\rangle_{\T},\
\forall\,\mathbf w\in \dot{\mathbf H}_{\#\sigma}^{0}.
\end{align}
The associated pressure $p$ is a distribution on $\T\times(0,T)$ satisfying \eqref{NS-problem-div0} in the sense of distributions, 
i.e.,
\begin{align}
\label{NS-eq:mik51a-d}
\left\langle {\mathbf u}'(\cdot,t)+({\mathbf u}(\cdot,t)\cdot \nabla ){\mathbf u}(\cdot,t),{\mathbf w}\right\rangle _{\T }
+a_{\T}({\mathbf u}(\cdot,t),{\mathbf w})
&+\langle\nabla p(\cdot,t),{\mathbf w}\rangle _{\T }
\nonumber\\
&=\langle\mathbf{f}(\cdot,t), \mathbf w\rangle_{\T},\ 
\text{for a.e. }  t\in(0,T),\ 
\forall\,\mathbf w\in {\mathbf C}^\infty_{\#}.
\end{align}
\end{definition}

To justify the weak formulation \eqref{NS-eq:mik51a}, let us act on \eqref{NS-problem-div0} by the Leray projector $\mathbb P_\sigma$ and taking into account that $\mathbb P_\sigma\partial_t\mathbf u=\partial_t\mathbf u$ and $\mathbb P_\sigma\nabla p=\mathbf 0$, we obtain
\begin{align}
\label{NS-problem-just}
\partial_t\mathbf u +\mathbb P_\sigma[({\mathbf u}\cdot \nabla ){\mathbf u}]-\mathbb P_\sigma\bs{\mathfrak L}{\mathbf u}&=\mathbb P_\sigma\mathbf{f}
\quad \mbox{in } \T\times(0,T).
\end{align}
Assuming that ${\mathbf u}\in  L_2(0,T;\dot{\mathbf H}_{\#\sigma}^{1})$, $a_{ij}^{\alpha \beta }\in L_\infty(0,T;L_{\infty\#})$,  by \eqref{Stokes-0} we obtain that $\bs{\mathfrak L}{\mathbf u}\in L_2(0,T;\dot{\mathbf H}_{\#}^{-1})$ and due to the symmetry conditions \eqref{Stokes-sym}, we get for any 
$\mathbf w\in \dot{\mathbf H}_{\#\sigma}^{1}$ and for a.e.  $t\in(0,T)$,
$$
-\langle\mathbb P_\sigma\bs{\mathfrak L}{\mathbf u}, \mathbf w\rangle_\T
=-\langle\bs{\mathfrak L}{\mathbf u}, \mathbf w\rangle_\T
=\langle a_{kj}^{\alpha \beta }(\cdot,t)E_{j\beta }({\mathbf u}),\partial_\alpha {\mathbf w}_k\rangle _{\T}
=\langle a_{kj}^{\alpha \beta }(\cdot,t)E_{j\beta }({\mathbf u}),E_{k,\alpha} ({\mathbf w})\rangle _{\T}
=a_{\T}({\mathbf u},{\mathbf w}).
$$
For ${\mathbf f}\in L_2(0,T;\dot{\mathbf H}_\#^{-1})$, we also have 
$\langle\mathbb P_\sigma\mathbf{f}(\cdot,t), \mathbf w\rangle_{\T}=\langle\mathbf{f}(\cdot,t), \mathbf w\rangle_{\T}$.
Hence taking the dual product of equation \eqref{NS-problem-just} with $\mathbf w$, we arrive at equation \eqref{NS-eq:mik51a}.  The boundedness of the first dual product in \eqref{NS-eq:mik51a} and the weak initial condition \eqref{NS-eq:mik51in} are justified in Lemma \ref{L4.2} and Remark \ref{R4.3} below.
Equation \eqref{NS-eq:mik51a-d} is deduced in a similar way.

\begin{lemma}\label{L4.2}
Let $n\ge 2$, $T>0$, $a_{ij}^{\alpha \beta }\in L_\infty(0,T;L_{\infty\#})$, ${\mathbf f}\in L_2(0,T;\dot{\mathbf H}_\#^{-1})$ and  
$\mathbf u^0\in \dot{\mathbf H}_{\#\sigma}^{0}$.
Let ${\mathbf u}\in  L_{\infty}(0,T;\dot{\mathbf H}_{\#\sigma}^{0})\cap L_2(0,T;\dot{\mathbf H}_{\#\sigma}^{1})$ solve equation  \eqref{NS-eq:mik51a}. 

(i) Then 
\begin{align}\label{E4.13}
\mathbf{Du}:={\mathbf u}'+\mathbb P_\sigma[({\mathbf u}\cdot \nabla ){\mathbf u}]\in L_2(0,T;\dot{\mathbf H}_{\#\sigma}^{-1})\quad
\text{and}\quad \mathbf{Du}(\cdot,t)\in \dot{\mathbf H}_{\#\sigma}^{-1} \text{ for a.e. } t\in [0,T],
\end{align}
while 
\begin{align}\label{E4.13A}
&({\mathbf u}\cdot \nabla ){\mathbf u}\in L_2(0,T;\dot{\mathbf H}_{\#}^{-n/2})\quad\text{and}\quad 
({\mathbf u}\cdot \nabla ){\mathbf u}(\cdot,t)\in \dot{\mathbf H}_{\#}^{-n/2} \quad \text{ for a.e. } t\in [0,T],
\\
\label{E4.13B}
&
{\mathbf u}'\in L_2(0,T;\dot{\mathbf H}_{\#\sigma}^{-n/2})
\quad \text{and}\quad 
{\mathbf u}'(\cdot,t)\in \dot{\mathbf H}_{\#\sigma}^{-n/2}\quad \text{ for a.e. } t\in [0,T],
\end{align}
and hence ${\mathbf u}\in W^1(\dot{\mathbf H}_{\#\sigma}^{1},\dot{\mathbf H}_{\#\sigma}^{-n/2})$.

In addition,
\begin{align}
\label{E4.29-3an}
&\partial_t\|{\mathbf u}\|_{\dot{\mathbf H}^{{-(n-2)/4}}_{\#\sigma}}^2
=2\langle \Lambda_\#^{-n/2} \mathbf u', \Lambda_\# {\mathbf u}\rangle _{\T } 
=2\langle \mathbf u',\Lambda_\#^{1-n/2}{\mathbf u}\rangle _{\T } 
=2\langle \Lambda_\#^{1-n/2}\mathbf u',{\mathbf u}\rangle _{\T } 
\end{align}
for a.e. $t\in(0,T)$ and also in the distribution sense on $(0,T)$.

(ii) Moreover, $\mathbf u$ is almost everywhere on $[0,T]$ equal to a function $\widetilde{\mathbf u}\in \mathcal C^0([0,T];\dot{\mathbf H}_{\#\sigma}^{-(n-2)/4})$,
and $\widetilde{\mathbf u}$ is also $\dot{\mathbf H}^0_{\#\sigma}$-weakly continuous in time on $[0,T]$, that is,
$$
\lim_{t\to t_0}\langle\widetilde{\mathbf u}(\cdot,t),\mathbf w\rangle_\T=\langle\widetilde{\mathbf u}(\cdot,t_0),\mathbf w\rangle_\T\quad \forall\, \mathbf w\in \mathbf H^0_{\#}, \ \forall\, t_0\in[0,T].
$$

(iii) There exists the associated pressure $p\in L_2(0,T;\dot{H}_{\#}^{-n/2+1})$ that for the given $\mathbf u$ is the unique solution of equation \eqref{NS-problem-div0} in this space. 
\end{lemma}
\begin{proof}
(i) By  \eqref{NS-eq:mik51a} we obtain
\begin{multline*}
\left|\langle \mathbf{Du}(\cdot,t),{\mathbf w}\rangle _{\T }\right|\le
|a_{\T}(t;{\mathbf u},{\mathbf w})| + |\langle\mathbf{f}(\cdot,t), \mathbf w\rangle_{\T}|
\\
\le \|\mathbb A\| \|\mathbf u(\cdot,t)\|_{{\mathbf H}_{\#}^1} \|\mathbf w\|_{{\mathbf H}_{\#}^1}
+\|\mathbf{f}(\cdot,t)\|_{{\mathbf H}_{\#}^{-1}} \|\mathbf w\|_{{\mathbf H}_{\#}^1},\ 
\text{for a.e. } t\in(0,T),\
\forall\,\mathbf w\in \dot{\mathbf H}_{\#\sigma}^{1}.
\end{multline*}
In addition, $\div\,\mathbf{Du}:=\div\,{\mathbf u}'+\div\,\mathbb P_\sigma[({\mathbf u}\cdot \nabla ){\mathbf u}]=0$.
Hence
$\|\mathbf{Du}(\cdot,t)\|_{\dot{\mathbf H}_{\#\sigma}^{-1}}\le \|\mathbb A\| \|\mathbf u(\cdot,t)\|_{{\mathbf H}_{\#}^1}
+\|\mathbf{f}(\cdot,t)\|_{{\mathbf H}_{\#}^{-1}}$ for a.e.  $t\in(0,T)$ and thus\\
$$\|\mathbf{Du}\|_{L_2(0,T;\dot{\mathbf H}_{\#\sigma}^{-1})}\le \|\mathbb A\| \|\mathbf u\|_{L_2(0,T;{\mathbf H}_{\#}^1)}
+\|\mathbf{f}\|_{L_2(0,T;{\mathbf H}_{\#}^{-1})}$$
which implies inclusions \eqref{E4.13}.

By  the multiplication Theorem \ref{RS-T1-S4.6.1} and the Sobolev interpolation inequality, we obtain
\begin{multline}\label{E4.37A}
\|(\mathbf u\cdot \nabla )\mathbf u\|_{\dot{\mathbf H}_{\#}^{-n/2}}
=\left\|\div(\mathbf u\otimes \mathbf u)\right\|_{{\mathbf H}_{\#}^{-n/2}}
\le \left\|\mathbf u\otimes \mathbf u\right\|_{(H_{\#}^{1-n/2})^{n\times n}}\\
\le C_*(1/2,1/2,n)\|\mathbf u\|^2_{{\mathbf H}_{\#}^{1/2}}
\le C_*(1/2,1/2,n)\|\mathbf u\|_{{\mathbf H}_{\#}^{0}}\|\mathbf u\|_{{\mathbf H}_{\#}^1}.
\end{multline}
Thus
$$
\|({\mathbf u}\cdot \nabla ){\mathbf u}\|_{L_2(0,T;\dot{\mathbf H}_{\#}^{-n/2})}
\le C_*(1/2,1/2,n)\|\mathbf u\|_{L_\infty(0,T;{\mathbf H}_{\#}^{0})}\|\mathbf u\|_{L_2(0,T;{\mathbf H}_{\#}^1)},
$$
which implies inclusions \eqref{E4.13A}.
Further,
\begin{align*}
\|\mathbb P_\sigma [({\mathbf u}\cdot \nabla ){\mathbf u}]\|_{L_2(0,T;\dot{\mathbf H}_{\#\sigma}^{-n/2})}
\le\|({\mathbf u}\cdot \nabla ){\mathbf u}\|_{L_2(0,T;\dot{\mathbf H}_{\#}^{-n/2})}
\le C_*(1/2,1/2,n)\|\mathbf u\|_{L_\infty(0,T;{\mathbf H}_{\#}^{0})}\|\mathbf u\|_{L_2(0,T;{\mathbf H}_{\#}^1)},
\end{align*}
implying that $\mathbb P_\sigma [({\mathbf u}\cdot \nabla ){\mathbf u}]\in{L_2(0,T;\dot{\mathbf H}_{\#\sigma}^{-n/2})}$. 
Then the first inclusion in \eqref{E4.13} leads to the inclusion ${\mathbf u}'\in{L_2(0,T;\dot{\mathbf H}_{\#\sigma}^{-n/2})}$ and hence to  inclusions \eqref{E4.13B}.

(ii) Since $\mathbf u\in L_{2}(0,T;\dot{\mathbf H}_{\#\sigma}^{1})$ and $\mathbf u'\in L_{2}(0,T;\dot{\mathbf H}_{\#\sigma}^{-n/2})$, 
relations \eqref{E4.29-3an} are implied by Lemma \ref{L4.9}(i).
Moreover, Theorem \ref{LM-T3.1} implies that $\mathbf u$ is almost everywhere on $[0,T]$ equal to a function $\widetilde{\mathbf u}\in \mathcal C^0([0,T];\dot{\mathbf H}_{\#\sigma}^{-(n-2)/4})$.

We have that $\widetilde{\mathbf u} \in L_\infty(0,T;\dot{\mathbf H}_{\#\sigma}^{0})$, $\widetilde{\mathbf u}\in \mathcal C^0([0,T];\dot{\mathbf H}_{\#\sigma}^{-(n-2)/4})$ and 
$\dot{\mathbf H}_{\#\sigma}^{0}\subset \dot{\mathbf H}_{\#\sigma}^{-(n-2)/4}$ with continuous injection. 
Then Lemma \ref{L1.4Tem} (taken from \cite[Chapter 3, Lemma 1.4]{Temam2001})  implies that $\widetilde{\mathbf u}$ is 
$\dot{\mathbf H}_{\#\sigma}^{0}$-weakly continuous in time.

(iii) The associated pressure $p$ satisfies \eqref{NS-problem-div0} that after applying the projector $\mathbb P_g$ can be re-written as
\begin{align}
\label{Eq-p}
\nabla p&=\mathbf{F},
\end{align}
where
\begin{align}
\mathbf F:=\mathbb P_g[\mathbf{f} - \partial_t\mathbf u +\bs{\mathfrak L}{\mathbf u}-({\mathbf u}\cdot \nabla ){\mathbf u}]
=\mathbb P_g\mathbf{f} +\mathbb P_g\bs{\mathfrak L}{\mathbf u}-\mathbb P_g[({\mathbf u}\cdot \nabla ){\mathbf u}]
\in L_2(0,T;\dot{\mathbf H}_{\# g}^{-n/2})
\end{align}
due to the first inclusion in \eqref{E4.13A}.
By Lemma \ref{div-grad-is} for gradient, with $s=1-n/2$, equation \eqref{Eq-p} has a unique solution $p$ in 
$L_2(0,T;\dot{H}_{\#}^{-n/2+1})$.
\end{proof}

Note  that inclusions \eqref{E4.13} do not generally imply that ${\mathbf u}'(\cdot,t)$ and 
$\mathbb P_\sigma[({\mathbf u}\cdot \nabla ){\mathbf u}](\cdot,t)$ belong to $\dot{\mathbf H}_{\#\sigma}^{-1}$ for a.e. $t\in [0,T]$, but only that their sum does. 
This is why the first dual product in \eqref{NS-eq:mik51a} is not written as the sum of the two respective dual products.

\begin{remark}\label{R4.3}
The initial condition \eqref{NS-eq:mik51in} should be understood for the function $\mathbf u$ re-defined as the function 
$\widetilde{\mathbf u}$ that was introduced in Lemma \ref{L4.2}(ii) and is $\mathbf H^0_{\#}$-weakly continuous in time.
\end{remark}

\section{Existence for evolution spatially-periodic anisotropic Navier-Stokes problem}

In this section, we prove solution existence for the evolution anisotropic incompressible Navier-Stokes systems, accommodating to anisotropy, variable coefficients and arbitrary $n\ge 2$ the approaches presented, e.g., in \cite[Chapter 1, Section 6.5]{Lions1969}, \cite[Chapter 8]{Constantin-Foias1988}, \cite[Chapter 3]{Temam1995}, \cite[Chapter 3, Section 3]{Temam2001}, \cite[Section 4]{RRS2016} for the constant-coefficient isotropic Navier-Stokes equations. 

\begin{theorem}
\label{NS-problemTh-sigma}
Let $n\ge 2$ and $T>0$. 
Let $a_{ij}^{\alpha \beta }\in L_\infty(0,T;L_{\infty\#})$
and the relaxed ellipticity condition \eqref{mu} hold.  
Let ${\mathbf f}\in L_2(0,T;\dot{\mathbf H}_\#^{-1})$, $\mathbf u^0\in \dot{\mathbf H}_{\#\sigma}^{0}$.

(i) Then there exists a weak solution 
${\mathbf u}\in  L_{\infty}(0,T;\dot{\mathbf H}_{\#\sigma}^{0})\cap L_2(0,T;\dot{\mathbf H}_{\#\sigma}^{1})$
of the anisotropic Navier-Stokes 
initial value problem  \eqref{NS-problem-div0}--\eqref{NS-problem-div0-IC} in the sense of Definition \ref{D6.1}.
Particularly,
$
\lim_{t\to 0}\langle{\mathbf u}(\cdot,t),\mathbf v\rangle_\T
=\langle{\mathbf u}^0,\mathbf v\rangle_\T\quad \forall\, \mathbf v\in \dot{\mathbf H}^0_{\#\sigma}.
$
There exists also the unique pressure $p\in L_2(0,T;\dot{H}_{\#}^{-n/2+1})$ associated with the obtained $\mathbf u$, that is the solution of equation \eqref{NS-problem-div0} in $L_2(0,T;\dot{H}_{\#}^{-n/2+1})$. 

(ii) Moreover, $\mathbf u$ satisfies  the following (strong) energy inequality,
\begin{align}
\| {\mathbf u}(\cdot,t)\| _{\mathbf L_{2\#}}^2
+2\int_{t_0}^{t}a_{\T}({\mathbf u}(\cdot,\tau),{\mathbf u}(\cdot,\tau)) d\tau
\label{NS-eq:mik51ai=a}
\le  \| {\mathbf u}(\cdot,t_0)\| ^2_{\mathbf L_{2\#}}
+2\int_{t_0}^{t}\langle\mathbf{f}(\cdot,\tau), {\mathbf u}(\cdot,\tau)\rangle_{\T}d\tau
\end{align}
for any $[t_0,t]\subset[0,T]$. It particularly implies
the standard energy inequality,
\begin{align}
\label{E4.9}
\| {\mathbf u}(\cdot,t)\| _{\mathbf L_{2\#}}^2
+2\int_0^ta_{\T}({\mathbf u}(\cdot,\tau),{\mathbf u}(\cdot,\tau)) d\tau
\le \| {\mathbf u}^0\| ^2_{\mathbf L_{2\#}}
+2\int_0^t\langle\mathbf{f}(\cdot,\tau), {\mathbf u}(\cdot,\tau)\rangle_{\T}\,d\tau
\quad \forall\,
t\in[0,T].
\end{align}

\end{theorem}
\begin{proof}
We prove the solution existence using the Faedo-Galerkin algorithm, cf., e.g., \cite[Chapt. 6, Sections 3, 6]{Ladyzhenskaya1969}, \cite[Chapter 1, Section 6.4]{Lions1969}, \cite[Chapter 3, Section 3.3]{Temam1995},  \cite[Chapter 3, Section 3]{Temam2001}, \cite[Section 4]{RRS2016}.
\paragraph{(a)} Let $\{\mathbf w_l\}=\mathbf w_1, \mathbf w_2,\ldots,\mathbf w_l,\ldots$ be 
the sequence of real orthonormal eigenfunctions of the Bessel potential operator $\Lambda_\#$ in $\dot{\mathbf H}^0_{\#\sigma}$, 
see Appendix \ref{BPOS}. 
This sequence constitutes 
an orthonormal basis in $\dot{\mathbf H}^0_{\#\sigma}$
and
is similar to a periodic version of the special basis employed in \cite[Chapter 1, Corollary 6.1]{Lions1969}. 
It belongs to $\dot{\mathbf C}^\infty_{\#\sigma}$ 
and can be explicitly expressed in terms of the Fourier harmonics,
see Remark \ref{RE.2}. 
Such choice of the linear independent functions particularly facilitates the proof of existence for arbitrary dimension $n\ge 2$.  
Another possible choice is given by the eigenfunctions of the isotropic Stokes operator in $\dot{\mathbf H}^0_{\#\sigma}$, cf. \cite[Section 2.2]{Temam1995}, \cite[Theorem 2.24]{RRS2016}.

For each integer $m\ge 1$, let us look for a solution
\begin{align}\label{NS-E5.14}
\mathbf u_m(\mathbf x,t)=\sum_{l=1}^m\eta_{l,m}(t)\mathbf w_l, \quad \eta_{l,m}(t)\in\R,
\end{align}
of the following discrete analogue of the initial-variational problem \eqref{NS-eq:mik51a}--\eqref{NS-eq:mik51in},
\begin{align}
\label{NS-eq:mik51ad}
&\langle {\mathbf u}'_m,{\mathbf w}_k\rangle _{\T }
+a_{\T}({\mathbf u}_m,{\bf w}_k)
+\langle({\mathbf u}_m\cdot \nabla ){\mathbf u}_m,{\bf w}_k\rangle _{\T}
=\langle\mathbf{f}, \mathbf w_k\rangle_{\T},\ \text{a.e. } t\in(0,T),\ \forall\, k\in \{1,\ldots,m\},\\
\label{NS-eq:mik51ind}
&\langle\mathbf u_m, \mathbf w_k\rangle_{\T}(\cdot,0)=\langle\mathbf u^0, \mathbf w_k\rangle_{\T},\quad \forall\, k\in \{1,\ldots,m\}.
\end{align} 
For a fixed $m$, equations \eqref{NS-eq:mik51ad}--\eqref{NS-eq:mik51ind} constitute an initial value problem for the nonlinear system of ordinary differential equations 
for unknowns $\eta_{l,m}(t)$,  $\ell\in \{1,\ldots,m\}$,
\begin{align}
\label{NS-eq:mik51ae}
&\sum_{l=1}^m\langle \mathbf w_l,{\mathbf w}_k\rangle _{\T}\,\partial_t\eta_{l,m}(t)
+\sum_{l=1}^m a_{\T}(t;{\mathbf w}_l,{\bf w}_k)\,\eta_{l,m}(t)
+\sum_{l,j=1}^m\langle({\mathbf w}_l\cdot \nabla ){\mathbf w}_j,{\bf w}_k\rangle _{\T}\eta_{l,m}(t)\eta_{j,m}(t)
\nonumber\\
&\hspace{21em}
=\langle\mathbf{f}, \mathbf w_k\rangle_{\T}\,,\ \text{a.e. } t\in(0,T),\quad \forall\, k\in \{1,\ldots,m\},\\
\label{NS-eq:mik51ine}
&\sum_{l=1}^m\langle\mathbf w_l, \mathbf w_k\rangle_{\T}\,\eta_{l,m}(0)=\langle\mathbf u^0, \mathbf w_k\rangle_{\T},\quad \forall\, k\in \{1,\ldots,m\}.
\end{align}
We have $\langle\mathbf{f}, \mathbf w_k\rangle_{\T}\in L_2(0,T)$ and
due to the orthonormality of the functions $\mathbf w_l$, we have $\langle \mathbf w_l,{\mathbf w}_k\rangle _{\T}=\delta_{\bs\ell k}$. 
Then by the Carath\'eodory existence theorem, see, e.g. \cite[Theorem 5.1]{Hale1980}, the ODE initial value problem \eqref{NS-eq:mik51ae}--\eqref{NS-eq:mik51ine} has an absolutely continuous solution $\eta_{l,m}(t)$, $l=1,\ldots,m$, on an interval $[0,T_m]$, $0<T_m\le T$. 

Multiplying equations \eqref{NS-eq:mik51ae} by $\eta_{k,m}$ and summing them up over $k\in \{1,\ldots,m\}$, and also doing the same with equations  \eqref{NS-eq:mik51ine}, we obtain 
\begin{align}
\label{NS-eq:mik51af}
&\langle \partial_t{\mathbf u}_m,{\mathbf u}_m\rangle _{\T }
+a_{\T}({\mathbf u}_m,{\mathbf u}_m)
+\langle({\mathbf u}_m\cdot \nabla ){\mathbf u}_m, {\mathbf u}_m\rangle_{\T}
=\langle\mathbf{f}, {\mathbf u}_m\rangle_{\T},\ \text{a.e. } t\in(0,T_m),\\
\label{NS-eq:mik51inf}
&\langle\mathbf u_m(\cdot,0), {\mathbf u}_m(\cdot,0)\rangle_{\T}=\langle\mathbf u^0, {\mathbf u}_m(\cdot,0)\rangle_{\T}.
\end{align} 
By equality \eqref{eq:mik55} for the trilinear term, equation \eqref{NS-eq:mik51af} is reduced to  
\begin{align}
\label{NS-eq:mik51af=}
&\frac12\partial_t\| {\mathbf u}_m\| _{\mathbf L_{2\#}}^2
+a_{\T}({\mathbf u}_m,{\mathbf u}_m)
=\langle\mathbf{f}, {\mathbf u}_m\rangle_{\T},\ \text{a.e. } t\in(0,T_m),
\end{align}

Inequality \eqref{NS-a-1-v2-S-}  for the quadratic form $a_{\T}$ and Yong's inequality for the right hand side of \eqref{NS-eq:mik51af=} imply
\begin{align}
\label{NS-eq:mik51ag}
\partial_t\| {\mathbf u}_m\| _{\mathbf L_{2\#}}^2
+\frac12 C_{\mathbb A}^{-1}\|{\mathbf u}_m\|_{\dot{\mathbf H}_\#^1}^2
\le\partial_t\| {\mathbf u}_m\| _{\mathbf L_{2\#}}^2
&+2a_{\T}({\mathbf u}_m,{\mathbf u}_m)
=2\langle\mathbf{f}, {\mathbf u}_m\rangle_{\T}
\le 2\|\mathbf f\|_{\dot{\mathbf H}_\#^{-1}} \|\mathbf u_m\|_{\dot{\mathbf H}_\#^1}\nonumber\\
&\le  \frac14 C_{\mathbb A}^{-1}\|\mathbf u_m\|^2_{\dot{\mathbf H}_\#^1} +4C_{\mathbb A}\|\mathbf f\|^2_{\dot{\mathbf H}_\#^{-1}},\ \text{a.e. } t\in(0,T_m).
\end{align}
Equation \eqref{NS-eq:mik51inf} implies
\begin{align}
\label{NS-eq:mik51ing}
\| {\mathbf u}_m(\cdot,0)\| _{\mathbf L_{2\#}}^2
=\langle\mathbf u^0, {\mathbf u}_m(\cdot,0)\rangle_{\T}
\le \| {\mathbf u}^0\| _{\mathbf L_{2\#}} \| {\mathbf u}_m(\cdot,0)\| _{\mathbf L_{2\#}}.
\end{align} 

Hence \eqref{NS-eq:mik51ag} and \eqref{NS-eq:mik51ing} lead to
\begin{align}
\label{NS-eq:mik51ah}
&\partial_t\| {\mathbf u}_m\| _{\mathbf L_{2\#}}^2
+ \frac14 C_{\mathbb A}^{-1}\|{\mathbf u}_m\|_{\dot{\mathbf H}_\#^1}^2
\le 4 C_{\mathbb A}\|\mathbf f\|^2_{\dot{\mathbf H}_\#^{-1}},\ \text{a.e. } t\in(0,T_m),
\\
\label{NS-eq:mik51inh}
&\| {\mathbf u}_m(\cdot,0)\| _{\mathbf L_{2\#}}
\le \| {\mathbf u}^0\| _{\mathbf L_{2\#}}.
\end{align} 
Integrating \eqref{NS-eq:mik51ah}, we get
\begin{align}
\label{NS-eq:mik51ai}
\| {\mathbf u}_m(\cdot,t)\| _{\mathbf L_{2\#}}^2
+\frac14 C_{\mathbb A}^{-1}\int_0^t\|{\mathbf u}_m(\cdot,\tau)\|_{\dot{\mathbf H}_\#^1}^2 d\tau
&\le  \| {\mathbf u}_m(\cdot,0)\| ^2_{\mathbf L_{2\#}}
+4 C_{\mathbb A}\int_0^t\|\mathbf f(\cdot,\tau)\|^2_{\dot{\mathbf H}_\#^{-1}}d\tau
\nonumber\\
&\le  \| {\mathbf u}^0\| ^2_{\mathbf L_{2\#}}
+4 C_{\mathbb A}\|\mathbf f\|^2_{L_2(0,T;\dot{\mathbf H}_\#^{-1})},
\quad t\in[0,T_m].
\end{align}
Estimate \eqref{NS-eq:mik51ai} particularly implies that the ODE initial value problem \eqref{NS-eq:mik51ae}--\eqref{NS-eq:mik51ine} has an absolutely continuous solution $\eta_{l,m}(t)$, $l=1,\ldots,m$, on the whole interval $[0,T]$, where the right hand side $\mathbf f$ 
is prescribed, i.e., we can take $T_m=T$.

Hence from \eqref{NS-eq:mik51ai} we conclude that 
\begin{align}
\label{NS-eq:mik51aj}
&\| {\mathbf u}_m\| ^2_{L_\infty(0,T;\mathbf L_{2\#})}
=\sup_{t\in [0,T]}\| {\mathbf u}_m(\cdot,t)\| _{\mathbf L_{2\#}}^2
\le  \| {\mathbf u}^0\| ^2_{\mathbf L_{2\#}}
+4 C_{\mathbb A}\|\mathbf f\|^2_{L_2(0,T;\dot{\mathbf H}_\#^{-1})},\\
\label{NS-eq:mik51ak}
&\|\mathbf u_m\|^2_{L_2(0,T;\dot{\mathbf H}_\#^{1})}
\le 4C_{\mathbb A}\left( \| {\mathbf u}^0\| ^2_{\mathbf L_{2\#}}
+4 C_{\mathbb A}\|\mathbf f\|^2_{L_2(0,T;\dot{\mathbf H}_\#^{-1})}\right).
\end{align}

 Recall that $\| {\mathbf u}\| _{\mathbf L_{2\#}}=\|\mathbf u\|_{{\mathbf H}_\#^0}$, while 
$\|\mathbf u\|_{{\mathbf H}_\#^s}=\|\mathbf u\|_{\dot{\mathbf H}_{\#\sigma}^s}$ for $\mathbf u\in{\dot{\mathbf H}_{\#\sigma}^s}$.
Estimates \eqref{NS-eq:mik51aj} and \eqref{NS-eq:mik51ak} mean that the sequence $\{\mathbf u_m\}$ is bounded in 
$L_\infty(0,T;\dot{\mathbf H}_{\#\sigma}^{0}))$ and in 
$L_2(0,T;\dot{\mathbf H}_{\#\sigma}^{1})$, implying that the sequence has a subsequence still denoted as $\{\mathbf u_m\}$  that converges weakly in $L_2(0,T;\dot{\mathbf H}_{\#\sigma}^{1})$ and weakly-star in $L_\infty(0,T;\dot{\mathbf H}_{\#\sigma}^{0})$ to a function 
$\mathbf u \in L_\infty(0,T;\dot{\mathbf H}_{\#\sigma}^{0})\cap L_2(0,T;\dot{\mathbf H}_{\#\sigma}^{1})$.
Note also that inequality \eqref{NS-eq:mik51aj} implies
also that
\begin{align*}
&
\| {\mathbf u}(\cdot,t)\| _{\dot{\mathbf H}_{\#\sigma}^{0}}^2
\le  \| {\mathbf u}^0\| ^2_{\dot{\mathbf H}_{\#\sigma}^{0}}
+4 C_{\mathbb A}\|\mathbf f\|^2_{L_2(0,T;\dot{\mathbf H}_\#^{-1})}\,,\ \text{a.e. } t\in(0,T).
\end{align*}
\paragraph{(b)} Let us also prove that the sequence $\{\mathbf u'_m\}$ is bounded in $L_2(0,T;\dot{\mathbf H}_{\#\sigma}^{-n/2})$, cf. \cite[Chapter 1, Section 6.4]{Lions1969}. 
To this end, we multiply equations \eqref{NS-eq:mik51ad} by $\mathbf w_k$ and sum them up over $k\in \{1,\ldots,m\}$, to obtain
\begin{align}\label{E4.25Pm}
\mathbf u'_m -P_m\bs{\mathfrak L}\mathbf u_m+P_m[({\mathbf u}_m\cdot \nabla ){\mathbf u}_m]&=P_m\mathbf{f}\,,\ \text{a.e. } t\in(0,T),
\end{align}
where $P_m$ is the projector operator 
from ${\mathbf H}^{-n/2}_{\#\sigma}$ to ${\rm Span}\{\mathbf w_1,\ldots,\mathbf w_m\}$ 
defined in \eqref{m-Projector}
and we took into account that
\begin{align*}
P_m \mathbf u'_m=\sum_{k=1}^m\langle\mathbf u'_m,\mathbf w_k\rangle_\T\mathbf w_k
=\sum_{k=1}^m\sum_{l=1}^m\eta'_{l,m}(t)\langle\mathbf w_l,\mathbf w_k\rangle_\T\mathbf w_k
=\sum_{l=1}^m\eta'_{l,m}(t)\mathbf w_l
=\mathbf u'_m.
\end{align*}
Further, due to Theorem \ref{TE.1r}(iii), for any $\mathbf h\in{\mathbf H}_{\#}^{r}$, $r\in\R$, we have
\begin{align}\label{E4.26Pm}
\|P_m \mathbf h\|^2_{\dot{\mathbf H}_{\#\sigma}^{r}}
\le\left\|\mathbf h\right\|^2_{{\mathbf H}_{\#}^{r}}.
\end{align}
By \eqref{E4.26Pm}, \eqref{L-oper} and \eqref{TensNorm} we have
\begin{align*}
\|P_m \bs{\mathfrak L}\mathbf u_m\|^2_{\dot{\mathbf H}_{\#\sigma}^{-n/2}}
\le \|\bs{\mathfrak L}\mathbf u_m\|^2_{\dot{\mathbf H}_{\#\sigma}^{-n/2}}
\le \left\|\bs{\mathfrak L}\mathbf u_m\right\|^2_{{\mathbf H}_{\#}^{-1}}
\le \|\mathbb A\|^2\|\mathbf u_m\|^2_{{\mathbf H}_{\#}^1}
\end{align*}
and then by \eqref{NS-eq:mik51ak},
\begin{multline}\label{E4.27pm}
\|P_m \bs{\mathfrak L}\mathbf u_m\|^2_{L_2(0,T;\dot{\mathbf H}_{\#\sigma}^{-n/2})}
\le\|\bs{\mathfrak L}\mathbf u_m\|^2_{L_2(0,T;\dot{\mathbf H}_{\#\sigma}^{-1})}
\le \|\mathbb A\|^2\|\mathbf u_m\|^2_{L_2(0,T;{\mathbf H}_{\#}^1)}
\\
\le 4\|\mathbb A\|^2 C_{\mathbb A}\left( \| {\mathbf u}^0\| ^2_{\mathbf L_{2\#}}
+4C_{\mathbb A}\|\mathbf f\|^2_{L_2(0,T;\dot{\mathbf H}_\#^{-1})}\right).
\end{multline}
Next, by \eqref{E4.26Pm} we obtain
\begin{align}
\|P_m \mathbf f\|^2_{L_2(0,T;\dot{\mathbf H}_{\#\sigma}^{-n/2})}
\le\|\mathbf f\|^2_{L_2(0,T;\dot{\mathbf H}_{\#\sigma}^{-n/2})}
\le \left\|\mathbf f\right\|^2_{L_2(0,T;{\mathbf H}_{\#}^{-1})}.
\label{E4.29Pm}
\end{align}

For any $\mathbf v_1\in \mathbf H_{\#\sigma}^1$, $\mathbf v_2\in \mathbf H_{\#}^1$,
by  Theorem \ref{RS-T1-S4.6.1}(b) and the Sobolev interpolation inequality, we obtain
\begin{multline*}
\|(\mathbf v_1\cdot \nabla )\mathbf v_2\|^2_{\dot{\mathbf H}_{\#}^{-n/2}}
=\left\|\div(\mathbf v_1\otimes \mathbf v_2)\right\|^2_{{\mathbf H}_{\#}^{-n/2}}
\le \left\|\mathbf v_1\otimes \mathbf v_2\right\|^2_{(H_{\#}^{1-n/2})^{n\times n}}\\
\le C_*^2(1/2,1/2,n)\|\mathbf v_1\|^2_{{\mathbf H}_{\#}^{1/2}}\|\mathbf v_2\|^2_{{\mathbf H}_{\#}^{1/2}}
\le C_*^2(1/2,1/2,n)\|\mathbf v_1\|_{{\mathbf H}_{\#}^{0}}\|\mathbf v_1\|_{{\mathbf H}_{\#}^1}
\|\mathbf v_2\|_{{\mathbf H}_{\#}^{0}}\|\mathbf v_2\|_{{\mathbf H}_{\#}^1}.
\end{multline*}
Thus
\begin{align*}
\|P_m [({\mathbf u}_m\cdot \nabla ){\mathbf u}_m]\|^2_{\dot{\mathbf H}_{\#\sigma}^{-n/2}}
\le \left\|({\mathbf u}_m\cdot \nabla ){\mathbf u}_m\right\|^2_{{\mathbf H}_{\#}^{-n/2}}
\le C_*^2(1/2,1/2,n)\|\mathbf u\|^2_{{\mathbf H}_{\#}^{0}}\|\mathbf u\|^2_{{\mathbf H}_{\#}^1}
\end{align*}
and then by \eqref{NS-eq:mik51aj} and \eqref{NS-eq:mik51ak},
\begin{align}\label{E4.28Pm}
\|P_m [({\mathbf u}_m\cdot \nabla ){\mathbf u}_m]\|^2_{L_2(0,T;\dot{\mathbf H}_{\#\sigma}^{-n/2})}
&\le C_*^2(1/2,1/2,n)\|\mathbf u_m\|^2_{L_\infty(0,T;{\mathbf H}_{\#}^{0})}\|\mathbf u_m\|^2_{L_2(0,T;{\mathbf H}_{\#}^1)}
\nonumber\\
&\le 4C_*^2(1/2,1/2,n)C_{\mathbb A}\left( \| {\mathbf u}^0\| ^2_{\mathbf L_{2\#}}
+4 C_{\mathbb A}\|\mathbf f\|^2_{L_2(0,T;\dot{\mathbf H}_\#^{-1})}\right)^2.
\end{align}

Equation \eqref{E4.25Pm} and inequalities  \eqref{E4.27pm},  \eqref{E4.29Pm} and \eqref{E4.28Pm} imply that the sequence $\{\mathbf u'_m\}$ is bounded in $L_2(0,T;\dot{\mathbf H}_{\#\sigma}^{-n/2})$ and hence it has a subsequence converging to a function
${\mathbf u}^\dag\in L_2(0,T;\dot{\mathbf H}_{\#\sigma}^{-n/2})$ weakly in this space.

Let us prove that ${\mathbf u}'={\mathbf u}^\dag$.
Indeed, for any $\phi\in \mathcal C^\infty_c(0,T)$ and $\mathbf w\in\dot{\mathbf H}_{\#\sigma}^{n/2}$, evidently, 
$\mathbf v:=\mathbf w \phi\in L_2(0,T;\dot{\mathbf H}_{\#\sigma}^{n/2})=\left(L_2(0,T;\dot{\mathbf H}_{\#\sigma}^{-n/2})\right)^*$
and we have
\begin{multline}\label{E4.40A0}
\int_0^T \langle \mathbf u^\dag(\cdot,t), \mathbf w\rangle_\T \phi(t)dt
=\int_0^T \langle {\mathbf u}^\dag(\cdot,t), \mathbf v(\cdot,t)\rangle_\T dt
\\
=\int_0^T \langle \mathbf u^\dag(\cdot,t)-\mathbf u'_m(\cdot,t), \mathbf v(\cdot,t)\rangle_\T dt
+\int_0^T \langle \mathbf u'_m(\cdot,t), \mathbf w\rangle_\T \phi(t)dt.
\end{multline}
The first integral in the right hand side of \eqref{E4.40A0} tends to zero as $m\to \infty$ due to the weak convergence of $\mathbf u'_m$ to ${\mathbf u}^\dag$ in $L_2(0,T;\dot{\mathbf H}_{\#\sigma}^{n/2})$. 
For the second integral in the right hand side of \eqref{E4.40A0} we obtain,
\begin{multline}\label{E4.40A1}
\int_0^T \langle \mathbf u'_m(\cdot,t), \mathbf w\rangle_\T\phi(t) dt
=-\int_0^T \langle \mathbf u_m(\cdot,t), \mathbf w\rangle_\T\phi'(t) dt
\\
=\int_0^T \langle \mathbf u(\cdot,t)-\mathbf u_m(\cdot,t), \mathbf w\rangle_\T\phi'(t) dt
-\int_0^T \langle \mathbf u(\cdot,t), \mathbf w\rangle_\T\phi'(t) dt
\end{multline}
The first integral in the right hand side of \eqref{E4.40A1} tends to zero as $m\to \infty$ due to the weak convergence of $\mathbf u_m$ to ${\mathbf u}$ in $L_2(0,T;\dot{\mathbf H}_{\#\sigma}^{1})$.
Hence, taking the limits of \eqref{E4.40A0} and \eqref{E4.40A1} as $m\to \infty$, we obtain,
\begin{align*}
\int_0^T \langle \mathbf u^\dag(\cdot,t), \mathbf w\rangle_\T \phi(t)dt
=-\int_0^T \langle \mathbf u(\cdot,t), \mathbf w\rangle_\T\phi'(t) dt
=\int_0^T \partial_t\langle \mathbf u(\cdot,t), \mathbf w\rangle_\T\phi(t) dt,
\end{align*}
which implies that $\langle \mathbf u^\dag(\cdot,t), \mathbf w\rangle_\T$ is the distributional derivative in time of $\langle \mathbf u(\cdot,t), \mathbf w\rangle_\T$ and thus as in the proof of Lemma \eqref{L4.9}(ii) the time derivative commutates with the dual product over $\T$, leading to
\begin{align*}
 \langle \mathbf u'(\cdot,t), \mathbf w\rangle_\T
=\partial_t\langle \mathbf u(\cdot,t), \mathbf w\rangle_\T
= \langle \mathbf u^\dag(\cdot,t), \mathbf w\rangle_\T
\end{align*}
in the sense of distributions on $(0,T)$, for any $\mathbf w\in\dot{\mathbf H}_{\#\sigma}^{n/2}$.
Since $\mathbf w$ is an arbitrary test function in $\dot{\mathbf H}_{\#\sigma}^{n/2}$, this implies that 
 $\mathbf u'={\mathbf u}^\dag$ and hence $\mathbf u'\in L_2(0,T;\dot{\mathbf H}_{\#\sigma}^{-n/2})$.

Applying now Theorem \ref{T-A-L} (the Aubin-Lions lemma) with 
$G=\dot{\mathbf H}_{\#\sigma}^{1}$, 
$H=\dot{\mathbf H}_{\#\sigma}^{0}$,
$K=\dot{\mathbf H}_{\#\sigma}^{-n/2}$
and $p=q=2$,
we conclude that  the subsequence $\{\mathbf u_m\}$ can be chosen in such a way that it converges to 
$\mathbf u \in L_\infty(0,T;\dot{\mathbf H}_{\#\sigma}^{0})\cap L_2(0,T;\dot{\mathbf H}_{\#\sigma}^{1})$ also strongly in $L_2(0,T;\dot{\mathbf H}_{\#\sigma}^{0})$.

Since $\mathbf u\in L_{2}(0,T;\dot{\mathbf H}_{\#\sigma}^{1})$ and $\mathbf u'\in L_{2}(0,T;\dot{\mathbf H}_{\#\sigma}^{-n/2})$,  Theorem \ref{LM-T3.1} implies that $\mathbf u$ is almost everywhere on $[0,T]$ equal to a function belonging to  $C^0([0,T];\dot{\mathbf H}_{\#\sigma}^{-(n-2)/4})$.
Further on, under $\mathbf u$ we will understand the redefined (on a zero-measure set in $[0,T]$) function belonging to  $C^0([0,T];\dot{\mathbf H}_{\#\sigma}^{-(n-2)/4})$, which also means 
that $
\| \mathbf u(\cdot,t)-{\mathbf u}(\cdot,0)\| _{\dot{\mathbf H}_{\#\sigma}^{-(n-2)/4}}\to 0
$
as ${t\to 0}$.
Since $\mathbf u\in L_{\infty}(0,T;\dot{\mathbf H}_{\#\sigma}^{0})$ as well, Lemma \ref{L1.4Tem}  implies that ${\mathbf u}$ is 
$\dot{\mathbf H}_{\#\sigma}^{0}$-weakly continuous 
in time on $[0,T]$ and hence
$
\lim_{t\to 0}\langle{\mathbf u}(\cdot,t),\mathbf v\rangle_\T
=\langle{\mathbf u}(\cdot,0),\mathbf v\rangle_\T\quad \forall\, \mathbf v\in \dot{\mathbf H}^0_{\#\sigma}.
$

\paragraph{(c)} Let us prove that the limit function $\mathbf u$ solves the initial-variational problem \eqref{NS-eq:mik51a}--\eqref{NS-eq:mik51in}.
First of all, equality \eqref{E4.25Pm} and inequality \eqref{E4.27pm} 
imply that
\begin{align*}
\|\mathbf u'_m +P_m[({\mathbf u}_m\cdot \nabla ){\mathbf u}_m]\|^2_{L_2(0,T;{\mathbf H}_{\#\sigma}^{-1})}
&\le 2\|P_m\mathbf{f}\|^2_{L_2(0,T;{\mathbf H}_{\#}^{-1})} +2\|P_m\bs{\mathfrak L}\mathbf u_m\|^2_{L_2(0,T;{\mathbf H}_{\#}^{-1})}
\nonumber\\
&\le 2\|\mathbf{f}\|^2_{L_2(0,T;{\mathbf H}_{\#}^{-1})} 
+8\|\mathbb A\|^2 C_{\mathbb A}\left( \| {\mathbf u}^0\| ^2_{\mathbf L_{2\#}}
+4C_{\mathbb A}\|\mathbf f\|^2_{L_2(0,T;\dot{\mathbf H}_\#^{-1})}\right).
\end{align*}
Thus the sequence $P_m\mathbf D{\mathbf u}_m:=\mathbf u'_m +P_m[({\mathbf u}_m\cdot \nabla ){\mathbf u}_m]$ is bounded in $L_2(0,T;\dot{\mathbf H}_{\#\sigma}^{-1})$ and hence there exists a subsequence of  the sequence ${\mathbf u}_m$ such that the corresponding subsequence of the sequence $P_m\mathbf D{\mathbf u}_m$ weakly converges in this space to a distribution ${\mathbf U}\in L_2(0,T;\dot{\mathbf H}_{\#\sigma}^{-1})$.
Let us prove that ${\mathbf U}=\mathbf D{\mathbf u}:=\mathbf u' +\mathbb P_\sigma[({\mathbf u}\cdot \nabla ){\mathbf u}]$.
Indeed, for any $\phi\in L_2(0,T)$ and $\mathbf w\in\dot{\mathbf H}_{\#\sigma}^{n/2}$, evidently, 
$$
\mathbf v:=\mathbf w \phi\in L_2(0,T;\dot{\mathbf H}_{\#\sigma}^{n/2})=\left(L_2(0,T;\dot{\mathbf H}_{\#\sigma}^{-n/2})\right)^*
\subset L_2(0,T;\dot{\mathbf H}_{\#\sigma}^{1})=\left(L_2(0,T;\dot{\mathbf H}_{\#\sigma}^{-1})\right)^*,
$$
and we have

\begin{multline}\label{E4.40B}
\int_0^T \langle \mathbf {Du}-{\mathbf U}, \mathbf v\rangle_\T dt
=\int_0^T \langle 
\mathbf {Du}-P_m\mathbf D{\mathbf u}_m, \mathbf v\rangle_\T dt
+\int_0^T \langle P_m\mathbf D{\mathbf u}_m-{\mathbf U}, \mathbf v\rangle_\T dt
\\
=\int_0^T \langle \mathbf u'-\mathbf u'_m, \mathbf v\rangle_\T dt
+\int_0^T \langle \mathbb P_\sigma[({\mathbf u}\cdot \nabla ){\mathbf u}]-P_m[({\mathbf u}_m\cdot \nabla ){\mathbf u}_m], \mathbf v\rangle_\T dt
+\int_0^T \langle P_m\mathbf D{\mathbf u}_m -{\mathbf U}, \mathbf v\rangle_\T dt.
\end{multline}
The first and the last integrals in the right hand side of \eqref{E4.40B} tends to zero as $m\to \infty$ due to the weak convergence of $\mathbf u'_m$ to $\mathbf u'$  in $L_2(0,T;\dot{\mathbf H}_{\#\sigma}^{-n/2})$ and of $P_m\mathbf {Du}_m$ to ${\mathbf U}$ in $L_2(0,T;\dot{\mathbf H}_{\#\sigma}^{-1})$. 
For the middle integral in the right hand side of \eqref{E4.40B} we obtain, as in \cite[Section 6.4.4]{Lions1969}, for any  function $\mathbf w_k\in \dot{\mathbf C}_{\#\sigma}^{\infty}$ from our basis in $\dot{\mathbf H}_{\#\sigma}^{n/2}$, for $m\ge k$,
\begin{multline}\label{E4.43}
\int_0^T \langle P_m[({\mathbf u}_m\cdot \nabla ){\mathbf u}_m], \mathbf w_k\rangle_\T \phi(t)dt
=\int_0^T \langle ({\mathbf u}_m\cdot \nabla ){\mathbf u}_m, \mathbf w_k\rangle_\T \phi(t)dt
\\
=-\int_0^T \langle{\mathbf u}_m\cdot\nabla \mathbf w_k,{\mathbf u}_m\rangle_\T \phi(t)dt
=-\int_0^T \langle\mathbb P_\sigma[{\mathbf u}_m\cdot\nabla \mathbf w_k],{\mathbf u}_m\rangle_\T \phi(t)dt
\\
\to -\int_0^T \langle\mathbb P_\sigma[{\mathbf u}\cdot\nabla \mathbf w_k],{\mathbf u}\rangle_\T \phi(t)dt
=\int_0^T \langle\mathbb P_\sigma[({\mathbf u}\cdot \nabla ){\mathbf u}], \mathbf w_k\rangle_\T \phi(t)dt,
\quad m\to\infty
\end{multline}
by the strong convergence of $\{\mathbf u_m\}$ to $\mathbf u$ in in $L_2(0,T;\dot{\mathbf H}_{\#\sigma}^{0})$.
Since $\mathbb P_\sigma[({\mathbf u}\cdot \nabla ){\mathbf u}]$ belongs to $L_2(0,T;\dot{\mathbf H}_{\#\sigma}^{-n/2})$, 
$P_m[({\mathbf u}_m\cdot \nabla ){\mathbf u}_m]$ is uniformly bounded in this space
and $\{\mathbf w_k\}$ is a basis in $\dot{\mathbf H}_{\#\sigma}^{n/2}$, we conclude that the convergence in \eqref{E4.43} implies that 
$$
\int_0^T \langle P_m[({\mathbf u}_m\cdot \nabla ){\mathbf u}_m], \mathbf w\rangle_\T \phi(t) dt\to
\int_0^T \langle\mathbb P_\sigma[({\mathbf u}\cdot \nabla ){\mathbf u}], \mathbf w\rangle_\T \phi(t) dt,
\quad m\to\infty
$$
and thus
\begin{align*}
\int_0^T \langle \mathbf {Du}-{\mathbf U}, \mathbf w\rangle_\T \phi(t)dt=0\ \ \forall\, \phi\in L_2(0,T), \  
\ \forall\, \mathbf w\in\dot{\mathbf H}_{\#\sigma}^{n/2},
\end{align*}
implying that $\|\langle \mathbf {Du}-{\mathbf U}, \mathbf w\rangle_\T\|_{L_2(0,T)}=0$ for any 
$\mathbf w\in\dot{\mathbf H}_{\#\sigma}^{n/2}$ and thus $\langle \mathbf {Du}(\cdot,t)-{\mathbf U}(\cdot,t), \mathbf w\rangle_\T=0$ for a.e. $t\in(0,T)$. Choosing $\mathbf w=\Lambda_\#^n(\mathbf {Du}-{\mathbf U})$, we conclude that 
$\|\mathbf {Du}(\cdot,t)-{\mathbf U}(\cdot,t)\|_{\dot{\mathbf H}_{\#\sigma}^{-n/2}}=0$ for a.e. $t\in(0,T)$
and hence
$\|\mathbf {Du}-{\mathbf U}\|_{L_2(0,T;\dot{\mathbf H}_{\#\sigma}^{-n/2})}=0$, 
i.e., $\mathbf {Du}={\mathbf U}\in L_2(0,T;\dot{\mathbf H}_{\#\sigma}^{-1})$.

Now we continue reasoning as, e.g., in \cite[Chapter 1, Section 6.4.4]{Lions1969} to conclude that
the limit function $\mathbf u$ solves the initial-variational problem \eqref{NS-eq:mik51a}--\eqref{NS-eq:mik51in}.
Indeed, let us multiply equation \eqref{NS-eq:mik51ad} by an arbitrary $\phi\in L_2(0,T)$, integrate it in $t$ to obtain. 
\begin{align}
\label{NS-eq:mik51aC1}
\int_0^T\big[\left\langle {\mathbf u}'_m+\mathbb P_\sigma[({\mathbf u}_m\cdot \nabla ){\mathbf u}_m],{\mathbf w_k}\right\rangle _{\T}
+ a_{\T}({\mathbf u}_m,{\mathbf w_k})
- \langle\mathbf{f}, \mathbf w_k\rangle_{\T}\big]\phi(t)dt=0,\ \ 
\forall\, k\in \{1,2,\ldots\}.
\end{align}
To take the limit of \eqref{NS-eq:mik51aC1} as $m\to \infty$, we remark that the terms linearly depending on ${\mathbf u}_m$ tend to the corresponding terms for ${\mathbf u}$ due to the weak convergences discussed before.
 For the nonlinear term, by \eqref{eq:mik55} we have
\begin{multline*}
\int_0^T\left\langle\mathbb P_\sigma[({\mathbf u}_m\cdot \nabla ){\mathbf u}_m],{\mathbf w_k}\right\rangle _{\T}\phi(t)dt
=\int_0^T\left\langle({\mathbf u}_m\cdot \nabla ){\mathbf u}_m,{\mathbf w_k}\right\rangle _{\T}\phi(t)dt
=-\int_0^T\left\langle({\mathbf u}_m\cdot \nabla ){\mathbf w_k},{\mathbf u}_m\right\rangle _{\T}\phi(t)dt
\\
\to -\int_0^T\left\langle({\mathbf u}\cdot \nabla ){\mathbf w_k},{\mathbf u}\right\rangle _{\T}\phi(t)dt
=\int_0^T\left\langle({\mathbf u}\cdot \nabla ){\mathbf u},{\mathbf w_k}\right\rangle _{\T}\phi(t)dt
=\int_0^T\left\langle\mathbb P_\sigma[({\mathbf u}\cdot \nabla ){\mathbf u}],{\mathbf w_k}\right\rangle _{\T}\phi(t)dt,
\end{multline*}
where the limit is due to the strong convergence of ${\mathbf u}_m$ to ${\mathbf u}$ in $L_2(0,T;\dot{\mathbf H}_{\#\sigma}^{0})$ and the smoothness of ${\mathbf w_k}$.
Thus, we obtain
\begin{align}
\label{NS-eq:mik51aC}
\int_0^T\big[\left\langle {\mathbf u}'+\mathbb P_\sigma[({\mathbf u}\cdot \nabla ){\mathbf u}],{\mathbf w_k}\right\rangle _{\T}
+ a_{\T}({\mathbf u},{\mathbf w_k})
- \langle\mathbf{f}, \mathbf w_k\rangle_{\T}\big]\phi(t)dt=0,\ \ 
\forall\, k\in \{1,2,\ldots\}.
\end{align}
Since  $\mathbf {Du}={\mathbf u}'+\mathbb P_\sigma[({\mathbf u}\cdot \nabla ){\mathbf u}]\in L_2(0,T;\dot{\mathbf H}_{\#\sigma}^{-1})$ 
and $\{\mathbf w_k\}$ is a basis in $\dot{\mathbf H}_{\#\sigma}^{1}$, equation \eqref{NS-eq:mik51aC} implies that 
\begin{align}
\label{NS-eq:mik51aD}
\int_0^T\big[\left\langle {\mathbf u}'+\mathbb P_\sigma[({\mathbf u}\cdot \nabla ){\mathbf u}],{\mathbf w}\right\rangle _{\T}
+ a_{\T}({\mathbf u},{\mathbf w})
- \langle\mathbf{f}, \mathbf w\rangle_{\T}\big]\phi(t)dt=0,\ \ \forall\, \phi\in L_2(0,T), \  \ 
\forall\,\mathbf w\in \dot{\mathbf H}_{\#\sigma}^1.
\end{align}
Equation \eqref{NS-eq:mik51aD} means that
$$
\left\|\left\langle {\mathbf u}'+\mathbb P_\sigma[({\mathbf u}\cdot \nabla ){\mathbf u}],{\mathbf w}\right\rangle _{\T }
+a_{\T}({\mathbf u},{\mathbf w})
-\langle\mathbf{f}, \mathbf w\rangle_{\T}\right\|_{L_2(0,T)}=0\  \ 
\forall\,\mathbf w\in \dot{\mathbf H}_{\#\sigma}^1,
$$
which implies \eqref{NS-eq:mik51a}.

To prove \eqref{NS-eq:mik51in}, let us employ in \eqref{NS-eq:mik51aC1} an arbitrary $\phi\in \mathcal C^\infty[0,T]$ such that $\phi(T)=0$, integrate the first term by parts with account of \eqref{NS-eq:mik51ind} and take the limit as $m\to \infty$ to obtain
\begin{multline}
\label{NS-eq:mik51aE}
\int_0^T \big\{\left\langle {-\mathbf u}(\cdot,t),{\mathbf w_k}\right\rangle _{\T}\phi'(t) 
+\left\langle\mathbb P_\sigma[({\mathbf u}(\cdot,t)\cdot \nabla ){\mathbf u}(\cdot,t)]\phi(t),{\mathbf w_k}\right\rangle _{\T}
+a_{\T}({\mathbf u}(\cdot,t),{\mathbf w_k})\phi(t)
\\
- \langle\mathbf{f}(\cdot,t), \mathbf w_k\rangle_{\T}\phi(t)\big\}dt= \langle\mathbf u^0, \mathbf w_k\rangle_{\T}\phi(0),\  \ 
\forall\, k\in \{1,2,\ldots\}.
\end{multline}
Replacing in \eqref{NS-eq:mik51aE} $\mathbf u$ by its redefined version that is $\dot{\mathbf H}_{\#\sigma}^{0}$-weakly continuous in time (cf. the last paragraph of the step {\bf(b)})  and integrating by parts the first term in \eqref{NS-eq:mik51aE}, we get 
\begin{multline*}
\int_0^T\big\{\left\langle {\mathbf u}'(\cdot,t)+\mathbb P_\sigma[({\mathbf u}(\cdot,t)\cdot \nabla ){\mathbf u}(\cdot,t)],{\mathbf w_k}\right\rangle _{\T}
+ a_{\T}({\mathbf u}(\cdot,t),{\mathbf w_k})
\\
- \langle\mathbf{f}(\cdot,t), \mathbf w_k\rangle_{\T}\big\}\phi(t)dt
=\langle\mathbf u^0, \mathbf w_k\rangle_{\T}\phi(0)-\langle\mathbf u(\cdot,0), \mathbf w_k\rangle_{\T}\phi(0),\ \ 
\forall\, k\in \{1,2,\ldots\}.
\end{multline*}
Comparing with \eqref{NS-eq:mik51aD} and taking into account that $\phi(0)$ is arbitrary, we obtain that 
$\langle\mathbf u^0-\langle\mathbf u(\cdot,0), \mathbf w_k\rangle_{\T}$, and because $ \mathbf w_k$ is a basis in $\dot{\mathbf H}_{\#\sigma}^{0}$, we conclude that $\mathbf u^0=\mathbf u(\cdot,0)$ thus proving the initial condition  \eqref{NS-eq:mik51in}.

The existence and uniqueness of the associated pressure $p\in L_2(0,T;\dot{H}_{\#}^{-n/2+1})$ follows from Lemma \ref{L4.2}(iii).

\paragraph{(d)} Let us prove the (strong) energy inequality (cf., \cite[Chapter 3, Remark 4(ii)]{Temam2001}, \cite[Theorem 4.6]{RRS2016} and references therein, for the isotropic constant-coefficient case). 
Here we will generalise the proof of \cite[Theorem 4.6]{RRS2016}.
To this end, let us consider again the subsequence $\{\mathbf u_m\}$ from the previous step, which still satisfies equation \eqref{NS-eq:mik51af=}.
Let $0\le t_0<t\le T$.
Multiplying  \eqref{NS-eq:mik51af=} by 2 and integrating it time, we get
\begin{align}
\| {\mathbf u}_m(\cdot,t)\| _{\mathbf L_{2\#}}^2
+2\int_{t_0}^{t}a_{\T}(\tau;{\mathbf u}_m(\cdot,\tau),{\mathbf u}_m(\cdot,\tau)) d\tau
\label{NS-eq:mik51ai=}
=  \| {\mathbf u}_m(\cdot,t_0)\| ^2_{\mathbf L_{2\#}}
+2\int_{t_0}^{t}\langle\mathbf{f}(\cdot,\tau), {\mathbf u}_m(\cdot,\tau)\rangle_{\T}d\tau.
\end{align}
We would like to take limits of each term in \eqref{NS-eq:mik51ai=} as $m\to\infty$.
First of all, since ${\mathbf u}_m$ converges to ${\mathbf u}$ strongly in $L_2(0,T;\dot{\mathbf H}_{\#\sigma}^{0})$, we obtain that 
\begin{align}\label{E4.37a}
\| {\mathbf u}_m(\cdot,\tau)\| _{\mathbf L_{2\#}}^2\to\| {\mathbf u}(\cdot,\tau)\| _{\mathbf L_{2\#}}^2, \text{ for a.e. }\tau\in[0,T].
\end{align}
Further, since ${\mathbf u}_m$ converges to ${\mathbf u}$ weakly in $L_2(0,T;\dot{\mathbf H}^1_{\#\sigma})$ and 
$\mathbf f\in{L_2(0,T;\dot{\mathbf H}_{\#}^{-1})}\subset(L_2(0,T;\dot{\mathbf H}^1_{\#\sigma}))^*$, we have
\begin{align}\label{E4.38a}
\int_{t_0}^{t}\langle\mathbf{f}(\cdot,\tau), {\mathbf u}_m(\cdot,\tau)\rangle_{\T}d\tau
\to \int_{t_0}^{t}\langle\mathbf{f}(\cdot,\tau), {\mathbf u}(\cdot,\tau)\rangle_{\T}d\tau, \quad\forall\,[t_0,t]\subset[0,T]
\end{align}
Finally,  ${\mathbf u}_m$ converges to ${\mathbf u}$ weakly in $L_2(0,T;\dot{\mathbf H}^1_{\#\sigma})$
and 
\begin{align*}
|\!|\!|\mathbf w|\!|\!|_{L_2(t_0,t;\dot{\mathbf H}_{\#}^1)}:=\left(\int_{t_0}^{t}a_{\T}(\tau;{\bf w}(\cdot,\tau),{\mathbf w}(\cdot,\tau))d\tau\right)^{1/2}
\end{align*} 
is an equivalent norm  
in $L_2(t_0,t;\dot{\mathbf H}_{\#\sigma}^{1})$, see \eqref{eqnorm}.
Since $\mathbf u$ is a weak limit of $\mathbf u_m$ in ${L_2(t_0,t;\dot{\mathbf H}_{\#}^1)}$,  we have
(see, e.g., the Remark in Section 4.43 of \cite{Lusternik-Sobolev1975}) that
$$
|\!|\!|\mathbf u|\!|\!|^2_{L_2(t_0,t;\dot{\mathbf H}_{\#}^1)}
\le\underset{m\to\infty}{\lim\inf}\, |\!|\!|\mathbf u_m|\!|\!|^2_{L_2(t_0,t;\dot{\mathbf H}_{\#}^1)}.
$$
Hence
\begin{align}\label{eqnorm2}
\int_{t_0}^{t}a_{\T}(\tau;{\mathbf u}(\cdot,\tau),{\mathbf u}(\cdot,\tau))d\tau
\le
\underset{m\to\infty}{\lim\inf}\int_{t_0}^{t}a_{\T}(\tau;{\mathbf u}_m(\cdot,\tau),{\mathbf u}_m(\cdot,\tau))d\tau, \quad\forall\,[t_0,t]\subset[0,T].
\end{align} 
Taking $\underset{m\to\infty}{\lim\inf}$ from both sides of \eqref{NS-eq:mik51ai=}, due to \eqref{E4.37a}, \eqref{E4.38a} and \eqref{eqnorm2}, we obtain \eqref{NS-eq:mik51ai=a}
 for a.e. $[t_0,t]\subset[0,T]$.

Similar to the reasoning in the proof of Theorem 4.6 in \cite{RRS2016}, let us now prove that the (strong) energy inequality \eqref{NS-eq:mik51ai=a} holds also {\em for any} $[t_0,t]\subset[0,T]$.
Let us take some $t_0$ for which \eqref{NS-eq:mik51ai=a} holds for a.e. $t'>t_0$. 
Let us now choose {\em any} $t\in(t_0,T]$. 
Then there exists a sequence $t'_{i}\to t$ such that
\begin{align*}
\| {\mathbf u}(\cdot,t'_{i})\| _{\mathbf L_{2\#}}^2
+2\int_{t_0}^{t'_{i}}a_{\T}(\tau;{\mathbf u}(\cdot,\tau),{\mathbf u}(\cdot,\tau)) d\tau
\le  \| {\mathbf u}(\cdot,t_0)\| ^2_{\mathbf L_{2\#}}
+2\int_{t_0}^{t'_{i}}\langle\mathbf{f}(\cdot,\tau), {\mathbf u}(\cdot,\tau)\rangle_{\T}d\tau.
\end{align*}
Since ${\mathbf u}\in L_2(0,T;\dot{\mathbf H}^1_{\#\sigma})$, we have
$$
\int_{t_0}^{t'_{i}}a_{\T}(\tau;{\mathbf u}(\cdot,\tau),{\mathbf u}(\cdot,\tau)) d\tau
\to\int_{t_0}^{t}a_{\T}(\tau;{\mathbf u}(\cdot,\tau),{\mathbf u}(\cdot,\tau)) d\tau,\quad
\int_{t_0}^{t'_{i}}\langle\mathbf{f}(\cdot,\tau), {\mathbf u}(\cdot,\tau)\rangle_{\T}d\tau
\to \int_{t_0}^{t}\langle\mathbf{f}(\cdot,\tau), {\mathbf u}(\cdot,\tau)\rangle_{\T}d\tau.
$$
On the other hand, the ${\mathbf L_2}_{\#}$-weak continuity of $\mathbf u$ implies that 
$$
\| {\mathbf u}(\cdot,t)\| _{\mathbf L_{2\#}}^2\le\underset{t'_i\to t}{\lim\inf}\| {\mathbf u}(\cdot,t'_{i})\| _{\mathbf L_{2\#}}^2.
$$
Thus
\begin{multline*}
\| {\mathbf u}(\cdot,t)\| _{\mathbf L_{2\#}}^2
+2\int_{t_0}^{t}a_{\T}(\tau;{\mathbf u}(\cdot,\tau),{\mathbf u}(\cdot,\tau)) d\tau
\le\underset{t'_i\to t}{\lim\inf}\left(\| {\mathbf u}(\cdot,t'_{i})\| _{\mathbf L_{2\#}}^2
+2\int_{t_0}^{t'_{i}}a_{\T}(\tau;{\mathbf u}(\cdot,\tau),{\mathbf u}(\cdot,\tau)) d\tau\right)
\\
\le \underset{t'_i\to t}{\lim\inf}\left( \| {\mathbf u}(\cdot,t_0)\| ^2_{\mathbf L_{2\#}}
+2\int_{t_0}^{t'_{i}}\langle\mathbf{f}(\cdot,\tau), {\mathbf u}(\cdot,\tau)\rangle_{\T}d\tau\right)
= \| {\mathbf u}(\cdot,t_0)\| ^2_{\mathbf L_{2\#}}
+2\int_{t_0}^{t}\langle\mathbf{f}(\cdot,\tau), {\mathbf u}(\cdot,\tau)\rangle_{\T}d\tau
\end{multline*}
By a similar argument, we can take any $t_0$.
\end{proof}

\section{Auxiliary results}

\subsection{Advection term properties}\label{S7.2} 
The  divergence theorem and periodicity
imply the following identity for any ${\mathbf v}_1,{\mathbf v}_2,{\mathbf v}_3\in \mathbf C^\infty_\#$.
\begin{align}
\label{eq:mik54}
\left\langle({\mathbf v}_1\cdot \nabla ){\mathbf v}_2,{\mathbf v}_3\right\rangle _{\T}
&=\int_{\T}\nabla\cdot\left({\mathbf v}_1({\mathbf v}_2\cdot {\mathbf v}_3)\right)d{\mathbf x}
-\left\langle(\nabla \cdot{\mathbf v}_1){\mathbf v}_3
+({\mathbf v}_1\cdot \nabla ){\mathbf v}_3,{\mathbf v}_2\right\rangle _{\T}\nonumber
\\
&=
-\left\langle({\mathbf v}_1\cdot \nabla ){\mathbf v}_3,{\mathbf v}_2\right\rangle _{\T}
-\left\langle(\nabla \cdot{\mathbf v}_1){\mathbf v}_3,{\mathbf v}_2\right\rangle _{\T}
 \end{align}
Hence for any ${\mathbf v}_1,{\mathbf v}_2\in \mathbf C^\infty_\#$,
\begin{equation}
\label{eq:mik55a}
\left\langle ({\mathbf v}_1\cdot \nabla ){\mathbf v}_2,{\mathbf v}_2\right\rangle _{\T}
=-\frac12\left\langle(\nabla \cdot{\mathbf v}_1){\mathbf v}_2,{\mathbf v}_2\right\rangle _{\T}
=-\frac12\left\langle \div\,{\mathbf v}_1,|{\mathbf v}_2|^2\right\rangle _{\T}.
\end{equation}
In view of \eqref{eq:mik54} we obtain the identity
\begin{align}\label{E-B.20}
\left\langle({\mathbf v}_1\cdot \nabla ){\mathbf v}_2,{\mathbf v}_3\right\rangle _{\T}
&\!\!=\!-\left\langle({\mathbf v}_1\cdot \nabla ){\mathbf v}_3,{\mathbf v}_2\right\rangle _{\T}\quad
 \forall\ {\mathbf v}_1\in \mathbf C^\infty_{\#\sigma},\
{\mathbf v}_2,\, {\mathbf v}_3\in \mathbf C^\infty_\#\,,
\end{align}
and hence the following well known formula for any  ${\mathbf v}_1\in \mathbf C^\infty_{\#\sigma}$,
${\mathbf v}_2\in \mathbf C^\infty_\#$,
\begin{equation}
\label{eq:mik55}
\left\langle ({\mathbf v}_1\cdot \nabla ){\mathbf v}_2,{\mathbf v}_2\right\rangle _{\T}=0.
\end{equation}
Equation \eqref{eq:mik55} evidently holds also  for ${\mathbf v}_1$ and ${\mathbf v}_2$ from  the more general spaces, for which the left hand side in \eqref{eq:mik55}  is bounded and to which  $\mathbf C^\infty_{\#\sigma}$ and $\mathbf C^\infty_\#$, respectively, are densely embedded.

\subsection{Some point-wise multiplication results}\label{S-RS}  
Let us  accommodated to the periodic function spaces in $\R^n$, $n\ge 1$, a particular case of a much more general Theorem 1 in Section 4.6.1 of \cite{Runst-Sickel1996} about point-wise products of functions/distributions.
\begin{theorem}\label{RS-T1-S4.6.1}
Assume $n\ge 1$, $ s_1\le s_2$ and $s_1+s_2>0$.
Then there exists a constant $C_*(s_1,s_2,n)>0$ such that for any $f_1\in H^{s_1}_\#$ and $f_2\in H^{s_1}_\#$, 

(a)  \quad 
$
f_1\cdot f_2\in H^{s_1}_\# \text{ and } \|f_1\cdot f_2\|_{H^{s_1}_\#}\le  C_*(s_1,s_2,n) \|f_1\|_{H^{s_1}_\#}\, \|f_2\|_{H^{s_2}_\#}
\quad\text{if }s_2>n/2;
$

(b) \quad
$
f_1\cdot f_2\in H^{s_1+s_2-n/2}_\# \text{ and } \|f_1\cdot f_2\|_{H^{s_1+s_2-n/2}_\#}\le  C_* (s_1,s_2,n)\|f_1\|_{H^{s_1}_\#}\, \|f_2\|_{H^{s_2}_\#}
 \quad \text{if }s_2<n/2.
$
\end{theorem}
\begin{proof}
Items (a) and (b) follow, respectively, from items (i) and (iii) of \cite[Theorem 1 in Section 4.6.1]{Runst-Sickel1996} when we take into account the norm equivalence in the standard and periodic Sobolev spaces.
\end{proof}

\subsection{Spectrum of the periodic Bessel potential operator}\label{BPOS}
In this section we assume that vector-functions/distributions $\mathbf u$ are generally complex-valued and the Sobolev spaces $\dot{\mathbf H}_{\#\sigma}^{s}$ are complex.
Let us recall the definition 
\begin{align}\label{E3.14sigma}
 \left(\Lambda_\#^{r}\,{\mathbf u}\right)(\mathbf x)
:=\sum_{\bs\xi\in\dot\Z^n}\varrho(\bs\xi)^{r}\widehat {\mathbf u}(\bs\xi)e^{2\pi i \mathbf x\cdot\bs\xi}\quad 
\forall\, \mathbf u\in  \dot{\mathbf H}_{\#\sigma}^{s},\ s\,r in\R.
\end{align}
of  the continuous periodic Bessel potential operator 
$\Lambda_\#^r: \dot{\mathbf H}_{\#\sigma}^s\to \dot{\mathbf H}_{\#\sigma}^{s-r}$, $r\in\R$,
see \eqref{E3.14}, \eqref{E3.15}, \eqref{3.23}.
\begin{theorem}\label{TE.1r}
Let $r\in\R$, $r\ne 0$.

(i) Then the operator $\Lambda^r_\#$ in  $\dot{\mathbf H}_{\#\sigma}^{0}$  possesses  a (non-strictly) monotone sequence of real eigenvalues $\lambda^{(r)}_j$ and a real orthonormal sequence of associated eigenfunctions $\mathbf w_j$ such that
\begin{align}
\label{S.2r}
&\Lambda^r_\# \mathbf w_j=\lambda^{(r)}_j \mathbf w_j,\  j \geq 1,\ 
\lambda^{(r)}_j>0,\\  
&\lambda^{(r)}_j \rightarrow+\infty,\  j \rightarrow+\infty \text{ if }r>0;
\quad \lambda^{(r)}_j \rightarrow 0,\  j \rightarrow+\infty \text{ if }r<0;
\\
&\mathbf{w}_j \in \dot{\mathbf C}_{\#\sigma}^{\infty}, \quad
(\mathbf w_j, \mathbf w_k)_{\dot{\mathbf H}_{\#\sigma}^{0}}
=\delta_{j k} \quad
\forall\, j, k >0.
\end{align}
(ii) Moreover, the sequence $\{\mathbf{w}_j\}$ is an orthonormal basis in  $\dot{\mathbf H}_{\#\sigma}^{0}$, that is 
\begin{align}\label{S.9ru2} 
\mathbf u=\sum_{i=1}^\infty \langle \mathbf u, \mathbf w_j\rangle_\T \mathbf w_j
\end{align}
where the series converges in $\dot{\mathbf H}_{\#\sigma}^{0}$ for any $\mathbf u\in \dot{\mathbf H}_{\#\sigma}^{0}$.

(iii) In addition, the sequence $\{\mathbf{w}_j\}$ is also an orthogonal basis in $\dot{\mathbf H}_{\#\sigma}^{r}$ with
$$
(\mathbf w_j, \mathbf w_k)_{\dot{\mathbf H}_{\#\sigma}^{r}}
=\lambda_j^{(r)}\lambda_k^{(r)}\delta_{jk}\quad 
\forall\, j, k >0.
$$
and  for any $\mathbf u\in \dot{\mathbf H}_{\#\sigma}^{r}$ series \eqref{S.9ru2}  converges also in $\dot{\mathbf H}_{\#\sigma}^{r}$, that is, the sequence of partial sums 
\begin{align}\label{m-Projector}
P_m\mathbf u:=\sum_{j=1}^m \langle\mathbf u,\mathbf w_j\rangle_\T \mathbf w_j
\end{align}
converges to $\mathbf u$ in $\dot{\mathbf H}_{\#\sigma}^{r}$ as $m\to \infty$.
The operator $P_m$ defined by \eqref{m-Projector} is for any $r\in\R$ the orthogonal projector operator from ${\mathbf H}^{r}_{\#}$ to ${\rm Span}\{\mathbf w_1,\ldots,\mathbf w_m\}$. 
\end{theorem}
\begin{proof}
Let 
first $r>0$ and let 
us consider the continuous periodic Bessel potential operator 
$\Lambda_\#^{-r}: \dot{\mathbf H}_{\#\sigma}^0\to \dot{\mathbf H}_{\#\sigma}^{r}$.
Hence by the Rellich-Kondrachov theorem the operator 
$\Lambda_\#^{-r}: \dot{\mathbf H}_{\#\sigma}^0\to \dot{\mathbf H}_{\#\sigma}^{0}$ is compact. 
It is also self-adjoint 
since for any $\mathbf u, \mathbf v\in  \dot{\mathbf H}_{\#\sigma}^{0}$ we have,

$$
(\Lambda_\#^{-r}\mathbf u, \mathbf v)_{\dot{\mathbf H}_{\#\sigma}^{0}}=\langle \Lambda_\#^{-r}\mathbf u, \bar{\mathbf v}\rangle_\T
=\langle \mathbf u, \Lambda_\#^{-r} \bar{\mathbf v}\rangle_\T=(\mathbf u, \Lambda_\#^{-r}\mathbf v)_{\dot{\mathbf H}_{\#\sigma}^{0}}.
$$

Then the Hilbert-Schmidt theorem (see, e.g., \cite[Theorem 8.94]{Renardy-Rogers2004}) implies that
there is a sequence of nonzero real eigenvalues $\left\{\lambda^{(-r)}_j\right\}_{i=1}^\infty$ of the operator 
$\Lambda_\#^{-r}: \dot{\mathbf H}_{\#\sigma}^0\to \dot{\mathbf H}_{\#\sigma}^{0}$, such that the sequence $\left|\lambda^{(-r)}_j\right|$ is monotone non-increasing and
$
\lim _{i \rightarrow \infty} \lambda^{(-r)}_j=0 .
$
Furthermore, if each eigenvalue of $\Lambda_\#^{-r}$ is repeated in the sequence according to its multiplicity, then  there exists an orthonormal (in $\dot{\mathbf H}_{\#\sigma}^0$) set $\left\{\mathbf w_j\right\}_{i=1}^\infty$ of the corresponding eigenfunctions, i.e.,
\begin{align}\label{S.6r}
\Lambda_\#^{-r} \mathbf w_j=\lambda^{(-r)}_j \mathbf w_j .
\end{align}
Moreover, the sequence $\left\{\mathbf w_j\right\}_{i=1}^\infty$ is an orthonormal basis in $\dot{\mathbf H}_{\#\sigma}^{0}$
for $\dot{\mathbf H}_{\#\sigma}^{r}$ as a subset of $\dot{\mathbf H}_{\#\sigma}^{0}$.

In addition, since the eigenvalues are real, \eqref{S.6r} implies that the eigenfunctions are either real or appear for the same eigenvalue in complex-conjugate pairs and hence their real and imaginary parts are also eigenfunctions. This means that we can choose the orthonormal basis consisting of real eigenfunctions only.

Since $\dot{\mathbf H}_{\#\sigma}^{r}$ is dense in  $\dot{\mathbf H}_{\#\sigma}^{0}$, the sequence $\left\{\mathbf w_j\right\}_{i=1}^\infty$ is an orthonormal basis  for  the entire space $\dot{\mathbf H}_{\#\sigma}^{0}$.
The operator $\Lambda_\#^{-r}$ can be represented as
\begin{align}\label{S.7r} 
\Lambda_\#^{-r} \mathbf v=\sum_{i=1}^\infty \lambda^{(-r)}_j\langle\mathbf v,\mathbf w_j\rangle_\T \mathbf w_j
\quad\forall\, \mathbf v\in \dot{\mathbf H}_{\#\sigma}^{0},
\end{align}
where the series converges in $\dot{\mathbf H}_{\#\sigma}^{0}$.

Let us remark that for any $\mathbf v\in \dot{\mathbf H}_{\#\sigma}^{0}$
\begin{align*}
(\Lambda^{-r}_\# \mathbf v,{\mathbf v})_{\dot{\mathbf H}_{\#\sigma}^{0}}
=\langle\Lambda^{-r}_\# \mathbf v,\bar{\mathbf v}\rangle_\T
=\langle\Lambda^{-r/2}_\# \mathbf v,\overline{\Lambda^{-r/2}_\#\mathbf v}\rangle_\T=
\|{\Lambda^{-r/2}_\#\mathbf v}\|^2_{\dot{\mathbf H}_{\#\sigma}^{0}}=\|\mathbf v\|^2_{\dot{\mathbf H}_{\#\sigma}^{r/2}}\ge\|\mathbf v\|^2_{\dot{\mathbf H}_{\#\sigma}^{0}},
\end{align*}
that is, $\Lambda^{-r}_\# $ is a positive-definite operator.
To conclude that all $\lambda_j$ are positive, we observe that for the unit real eigenfunctions $\mathbf w_j$, \eqref{S.6r} implies
\begin{align*}
\lambda^{(-r)}_j=\lambda^{(-r)}_j \langle\mathbf w_j,\mathbf w_j\rangle_\T=\langle\Lambda^{-r}_\# \mathbf w_j,\mathbf w_j\rangle_\T
=\langle\Lambda^{-r/2}_\# \mathbf w_j,\Lambda^{-r/2}_\#\mathbf w_j\rangle_\T>0.
\end{align*}

Applying $\Lambda^r$ to \eqref{S.6r}, we obtain
\begin{align}\label{S.8r}
\Lambda_\#^{r} \mathbf w_j=\lambda^{(r)}_j \mathbf w_j, \text{ where } \lambda_j^{(r)}=1/\lambda_j^{(-r)}
\end{align}
implying \eqref{S.2r} with $\lambda_j^{(r)}=1/\lambda_j^{(-r)}$ and the coinciding eigenfunctions for the operators $\Lambda_\#^{r}$ and $\Lambda_\#^{-r}$.  

Since $\mathbf w_j\in \dot{\mathbf H}_{\#\sigma}^{0}$ and $\lambda_j\ne 0$, equation
\eqref{S.6r} implies 
$\mathbf w_j\in \dot{\mathbf H}_{\#\sigma}^{r}$.
Moreover, applying to \eqref{S.6r} operator $\Lambda^{-r(k-1)}$, with any integer $k$, and employing consecutively \eqref{S.8r} or \eqref{S.6r}, we obtain
\begin{align}\label{S.9r}
\Lambda^{-rk}_\# \mathbf w_j=(\lambda^{(-r)}_j)^k \mathbf w_j\quad \forall\, k\in \Z.
\end{align}
and taking into account the continuity of the operator
$\Lambda_\#^{-rk}: \dot{\mathbf H}_{\#\sigma}^0\to \dot{\mathbf H}_{\#\sigma}^{kr}$ for any integer $k$, 
we conclude that $\mathbf w_j\in \dot{\mathbf C}_{\#\sigma}^{\infty}$.

Finally, let us prove that the sequence $\{\mathbf{w}_j\}$  is an orthogonal basis also in $\dot{\mathbf H}_{\#\sigma}^{r}$.
To this end, let $\mathbf u\in \dot{\mathbf H}_{\#\sigma}^{r}$. 
We know that the series \eqref{S.9ru2} 
converges in  $\dot{\mathbf H}_{\#\sigma}^{0}$. 
Let us prove that it converges also in $\dot{\mathbf H}_{\#\sigma}^{r}$, that is, the sequence of its partial sums converges in this space.
Indeed,
\begin{align}\label{S.10ru2} 
\sum_{j=1} \langle\mathbf u,\mathbf w_j\rangle_\T \mathbf w_j
=\sum_{j=1} \langle\mathbf u,\lambda^{(r)}_j\mathbf w_j\rangle_\T \lambda^{(-r)}_j\mathbf w_j
=\sum_{j=1} \langle\mathbf u,\Lambda^r_\#\mathbf w_j\rangle_\T \Lambda^{-r}_\#\mathbf w_j
=\Lambda^{-r}_\#\sum_{j=1} \langle\Lambda^r_\#\mathbf u,\mathbf w_j\rangle_\T \mathbf w_j
\end{align}
Since $\mathbf u\in \dot{\mathbf H}_{\#\sigma}^{r}$ we have that $\Lambda^r_\#\mathbf u\in\dot{\mathbf H}_{\#\sigma}^{0}$ implying that the sequence $\sum_{i=1} \langle\Lambda^r_\#\mathbf u,\mathbf w_j\rangle_\T \mathbf w_j$ converges in 
$\dot{\mathbf H}_{\#\sigma}^{0}$  to $\Lambda^r_\#\mathbf u$ as $m\to\infty$.
The continuity of the operator $\Lambda_\#^{-r}: \dot{\mathbf H}_{\#\sigma}^0\to \dot{\mathbf H}_{\#\sigma}^{r}$ then implies that the right hand side of  \eqref{S.10ru2} converges in $\dot{\mathbf H}_{\#\sigma}^{r}$ to $\mathbf u$ together with the sequence of the partial sums in the left hand side.
This means that series \eqref{S.9ru2} converges in $\dot{\mathbf H}_{\#\sigma}^{r}$ to $\mathbf u$ as well.
Thus the set $\{\mathbf{w}_j\}$  is complete in $\dot{\mathbf H}_{\#\sigma}^{r}$.

The orthogonality of the set $\{\mathbf{w}_j\}$ in $\dot{\mathbf H}_{\#\sigma}^{r}$ is implied by the relations
$$
(\mathbf w_j, \mathbf w_k)_{\dot{\mathbf H}_{\#\sigma}^{r}}
=(\Lambda_\#^{r}\mathbf w_j, \Lambda_\#^{r}\mathbf w_k)_{\dot{\mathbf H}_{\#\sigma}^{0}}
=(\lambda_j^{(r)}\mathbf w_j, \lambda_k^{(r)}\mathbf w_k)_{\dot{\mathbf H}_{\#\sigma}^{0}}
=\lambda_j^{(r)}\lambda_k^{(r)}\langle\mathbf w_j, \mathbf w_k\rangle_\T
=\lambda_j^{(r)}\lambda_k^{(r)}\delta_{jk}.
$$
Hence the set $\{\mathbf{w}_j\}$ is an orthogonal basis in $\dot{\mathbf H}_{\#\sigma}^{r}$.

Although we started from $r>0$, in the proof we covered the cases of both positive and negative $r$.
\end{proof}

Similar to the reasoning at the end of Section 2.2 in \cite{Temam1995}, for the eigenvalues and eigenfunctions of the isotropic Stokes operator in periodic setting, let us provide an explicit representation of the eigenvalues and eigenfunctions of the operator 
$\Lambda_\#^r: \dot{\mathbf H}_{\#\sigma}^0\to \dot{\mathbf H}_{\#\sigma}^{0}$, $r\in\R$, $r\ne 0$.

Employing representations \eqref{eq:mik11} and \eqref{E3.14sigma} in \eqref{S.2r}, we obtain for a fixed $j$,
\begin{align}\label{S-E.16}
\sum_{\bs\xi\in\dot\Z^n}\varrho(\bs\xi)^{r}\widehat{\mathbf w}_j(\bs\xi)e^{2\pi i \mathbf x\cdot\bs\xi}
=\lambda^{(r)}_j\sum_{\bs\xi\in\dot\Z^n}\widehat{\mathbf w}_j(\bs\xi)e^{2\pi i \mathbf x\cdot\bs\xi},
\end{align}
that is,
\begin{align}\label{S-E.16a}
\left(\varrho(\bs\xi)^{r}-\lambda^{(r)}_j\right)\widehat{\mathbf w}_j(\bs\xi)=0
\quad\forall\,\bs\xi\in\dot\Z^n .
\end{align}
This implies that the eigenvalues and the corresponding eigenfunctions of the operator $\Lambda^r$ can be explicitly represented as 
$\{\lambda^{(r)}_j\}=\{\lambda^{(r)}_{\bs\eta,\beta}\}$, $\{\mathbf w_j\}=\{\mathbf w_{\bs\eta,\beta}\}$, where $\bs\eta\in\dot\Z^n$, $\beta=\{1,\ldots,n-1\}$,
\begin{align}\label{S-E.17}
\lambda^{(r)}_{\bs\eta,\beta}=\varrho(\bs\eta)^r=(1+|\bs\eta|^2)^{r/2},\quad
{\mathbf w}_{\bs\eta,\beta}(\mathbf x)=\mathring{\mathbf w}_{\bs\eta,\beta}e^{2\pi i \mathbf x\cdot\bs\eta}.
\end{align}
For a fixed $\bs\eta$, the $n-1$ orthonormal constant real vectors $\mathring{\mathbf w}_{\bs\eta,\beta}$, $\beta=\{1,\ldots,n-1\}$ are obtained by the orthogonalisation in $\R^n$ of the  real vector set
$$
\widetilde{\mathbf w}_{\bs\eta,\alpha}=\mathbf e_\alpha-\frac{\eta_\alpha \bs\eta}{|\bs \eta|^2},\quad \alpha=\{1,\ldots,n\},
$$
where $\mathbf e_\alpha$ are canonical (coordinate) vectors in $\R^n$.
Note that $(\widetilde{\mathbf w}_{\bs\eta,\alpha}\cdot\bs\eta)=0$.

\begin{remark}\label{RE.2}
Relations \eqref{S-E.17} particularly imply that $\lambda^{(r)}_{\bs\eta,\beta}=\lambda^r_{\bs\eta,\beta}$, where $\lambda_{\bs\eta,\beta}:=\lambda^{(1)}_{\bs\eta,\beta}=\varrho(\bs\eta)=(1+|\bs\eta|^2)^{1/2}$, i.e., $\lambda^{(r)}_j=\lambda^r_j$ and the corresponding eigenfunctions coincide for any $r\in\R$, $r\ne 0$. Since the sequence of eigenfunctions  $\{\mathbf w_j\}$ corresponding to $\lambda^{(r)}_j$ is the same for any $r\in\R$,  $r\ne 0$, Theorem \ref{TE.1r} implies that the sequence constitutes a real orthogonal basis in any space $\dot{\mathbf H}_{\#\sigma}^{r}$, $r\in\R$.
\end{remark}

\subsection{Isomorphism of divergence and gradient operators in periodic spaces}
In the following assertion we provide for arbitrary $s\in\R$ and dimension $n\ge 2$ the periodic version of Bogovskii/deRham--type results well known for non-periodic domains and particular values of $s$, see, e.g., \cite{Bogovskii1979},  \cite{AmCiMa2015} and references therein.
\begin{lemma}\label{div-grad-is}
Let $s\in\R$ and $n\ge 2$. 
The operators 
\begin{align}
&\div\, :  \dot{\mathbf H}_{\# g}^{s+1}\to  \dot{H}_{\#}^{s},
\label{2.37}
\\
&{\rm grad}\, :  \dot{H}_{\#}^{s}\to  \dot{\mathbf H}_{\# g}^{s-1}
\label{2.38}
\end{align}
are isomorphisms.
\end{lemma}
\begin{proof}
(i) Since $\dot{\mathbf H}_{\# g}^{s+1}\subset \dot{\mathbf H}_{\#}^{s+1}$, operator \eqref{2.37} is continuous.
Let $f\in \dot{H}_{\#}^{s}$ and let us consider the equation 
\begin{align}\label{E2.39}
\div\,\mathbf F=f
\end{align}
for $\mathbf F\in \dot{\mathbf H}_{\# g}^{s+1}$.
Calculating the Fourier coefficients of both sides of the equation, we obtain
\begin{align*}
2\pi i\bs\xi\cdot\widehat{\mathbf F}(\bs\xi)=\widehat f(\bs\xi), \quad \bs\xi\in\dot\Z^n.
\end{align*}
By inspection one can see that this equation has a solution in the form
\begin{align}\label{E2.41}
\widehat{\mathbf F}(\bs\xi)=\frac{\bs\xi \widehat f(\bs\xi)}{2\pi i|\bs\xi|^2}, \quad \bs\xi\in\dot\Z^n,
\end{align}
that is,
\begin{align*}
\widehat{\mathbf F}(\bs\xi)=2\pi i \bs\xi\widehat q=\widehat{\nabla q},
\text { where }
\widehat q=-\frac{\widehat f(\bs\xi)}{(2\pi)^2|\bs\xi|^2}, \quad \bs\xi\in\dot\Z^n.
\end{align*}
By \eqref{E2.41}, \eqref{eq:mik9} and \eqref{eq:mik10}, we obtain
$$
\|\mathbf F\|^2_{\dot{\mathbf H}_{\#}^{s+1}}
=\sum_{\bs\xi\in\dot\Z^n}\varrho(\bs\xi)^{2(s+1)}|\widehat{\mathbf F}(\bs\xi)|^2
=\sum_{\bs\xi\in\dot\Z^n}\varrho(\bs\xi)^{2s}\frac{\varrho(\bs\xi)^{2}}{(2\pi)^2|\bs\xi|^2}|\widehat{f}(\bs\xi)|^2
\le 2\sum_{\bs\xi\in\dot\Z^n}\varrho(\bs\xi)^{2s}|\widehat{f}(\bs\xi)|^2
=2\|f\|^2_{\dot{H}_{\#}^{s}}.
$$
Hence the solution $\mathbf F$ given by \eqref{E2.41} belongs to $\dot{\mathbf H}_{\# g}^{s+1}$ and satisfies the estimate 
$
\|\mathbf F\|_{\dot{\mathbf H}_{\#}^{s+1}}\le\sqrt{2}\|f\|_{\dot{H}_{\#}^{s}}.
$
There are no other solutions in $\dot{\mathbf H}_{\# g}^{s+1}$ since otherwise the difference, $\widetilde{\mathbf F}$, of two solutions of equation \eqref{E2.39} would satisfy equation $\div\,\widetilde{\mathbf F}=0$,  and hence belong to $\dot{\mathbf H}_{\# g}^{s+1}\cap \dot{\mathbf H}_{\# \sigma}^{s+1}=\{\mathbf 0\}$.
Thus operator \eqref{2.37} is an isomorphism.

(ii)
By the definition of the space $\dot{\mathbf H}_{\# g}^{s-1}$, operator \eqref{2.38} is continuous.
Let $\mathbf F\in \dot{\mathbf H}_{\# g}^{s-1}$ and let us consider the equation 
\begin{align}\label{E2.42-0}
\nabla f=\mathbf F
\end{align}
for $f\in \dot{H}_{\#}^{s}$.
Equation \eqref{E2.42-0} has at most one solution since otherwise the difference of any two solutions, $\widetilde f$, would satisfy the equation 
$\nabla \widetilde f=\mathbf 0$ implying that $\widetilde f=const=0$ because $f\in \dot{H}_{\#}^{s}$.
Taking into account that $\mathbf F=\nabla q$ for some $q\in \dot{H}_{\#}^{s}$, we conclude that there exists a solution of equation \eqref{E2.42-0}, namely $f=q$. 

Let us calculate the norm estimate for this solution.
Calculating the Fourier coefficients of both sides of equation \eqref{E2.42-0}, we obtain
\begin{align}\label{E2.42-1}
2\pi i\bs\xi \widehat{f}(\bs\xi)=\widehat{\mathbf F}(\bs\xi), \quad \bs\xi\in\dot\Z^n.
\end{align}
Then
\begin{align}\label{E2.42}
\widehat f(\bs\xi)=\frac{\bs\xi\cdot\widehat{\mathbf F}(\bs\xi)}{2\pi i|\bs\xi|^2}, \quad \bs\xi\in\dot\Z^n.
\end{align}
By \eqref{E2.42}, \eqref{eq:mik9} and \eqref{eq:mik10}, we obtain
$$
\|f\|^2_{\dot{H}_{\#}^{s}}=\sum_{\bs\xi\in\dot\Z^n}\varrho(\bs\xi)^{2s}|\widehat{f}(\bs\xi)|^2
=\sum_{\bs\xi\in\dot\Z^n}\frac{\varrho(\bs\xi)^{2s}}{(2\pi)^2|\bs\xi|^4}|\bs\xi\cdot\widehat{\mathbf F}(\bs\xi)|^2
\le 2\sum_{\bs\xi\in\dot\Z^n}\varrho(\bs\xi)^{2(s-1)}|\widehat{\mathbf F}(\bs\xi)|^2
=2\|\mathbf F\|^2_{\dot{\mathbf H}_{\#}^{s-1}}.
$$
Hence the solution $f$ given by \eqref{E2.42} belongs to $\dot{H}_{\#}^{s}$ and satisfies the estimate 
$
\|f\|_{\dot{H}_{\#}^{s}}\le\sqrt{2}\|\mathbf F\|_{\dot{\mathbf H}_{\#}^{s-1}}.
$
Thus operator \eqref{2.38} is an isomorphism.
\end{proof}

\subsection{Some functional analysis results}
The Aubin–Lions Lemma, see 
\cite[Chapter 1, Theorem 5.1]{Lions1969}, has been generalised in \cite{Simon1987}.
We provide it in the form of Theorem 4.12 in \cite{RRS2016}.
\begin{theorem}[Aubin–Lions Lemma]\label{T-A-L} Suppose that $G \subset H \subset K$ where $G, H$ and $K$ are reflexive Banach spaces and the embedding $G \subset H$ is compact. Let $1\le p\le\infty$ and $1\le q <\infty$. If the sequence $u_n$ is bounded in $L_q(0, T ; G)$ and $\partial_t u_n$ is bounded in $L_{p}(0, T ; K)$, then there exists a subsequence of $u_n$ that is strongly convergent in $L_{q}(0, T ; H)$.
\end{theorem}

The following assertion is available in \cite[Chapter 3, Lemma 1.4]{Temam2001}
\begin{lemma}\label{L1.4Tem} Let $X$ and $Y$ be two Banach spaces, such that $X\subset Y$with a continuous injection. If a function $v$ belongs to $L_\infty (0, T; X)$ and is weakly continuous with values in $Y$, then $v$ is weakly continuous with values in $X$.
\end{lemma}

Theorem 3.1 and Remark 3.2 in Chapter 1 of \cite{Lions-Magenes1}  imply the following assertion.
\begin{theorem}\label{LM-T3.1}
Let $X$ and $Y$ be separable Hilbert spaces and $X\subset Y$ with continuous injection.
Let $u\in W^1(0,T;X, Y)$.
Than $u$ almost everywhere on $[0,T]$ equals to a function   
$\tilde u\in \mathcal C^0([0,T];Z)$, where $Z=[X,Y]_{1/2}$ is the intermediate space,
Then the trace $u(0)\in Z$ is defined as the corresponding value of $\tilde u\in \mathcal C^0([0,T];Z)$ at $t=0$.
\end{theorem}

Let us prove the following assertion inspired by  Lemmas 1.2 and 1.3 in Chapter 3 of  \cite{Temam2001}.
\begin{lemma}\label{L4.9}
Let $s,s'\in\R$, $s'\le s$ and $u\in W^1(0,T;H^s_\#,H^{s'}_\#)$ be real-valued. 

(i) Then
\begin{align}\label{E3.69}
\partial_t\|u\|^2_{H^{(s+s')/2}_\#}= 2\langle \Lambda_\#^{s'}u',\Lambda_\#^{s}u\rangle_{\T} = 2\langle \Lambda_\#^{s'+s}u',u\rangle_{\T} 
\end{align}
for a.e. $t\in(0,T)$ and also in the distribution sense on $t\in (0,T)$.

(ii) Moreover, for any real $v\in W^1(0,T;{  H}_{\#}^{-s'}, {  H}_{\#}^{-s})$ and $t\in(0,T]$,
\begin{align}
\label{E3.30}
\int_0^t\left[\langle {  u}' (\tau),{  v} (\tau)\rangle _{\T }
+\langle {  u} (\tau),{  v}' (\tau)\rangle _{\T }\right]d\tau
=\left\langle {  u}(t),{  v}(t)\right\rangle_{\T }
-\left\langle {  u}(0),{  v}(0)\right\rangle _{\T }.
\end{align}
\end{lemma} 
\begin{proof}
(i) Since $u\in W^1(0,T;H^s_\#,H^{s'}_\#)$,
there exists a sequence of infinitely differentiable functions $\{u_m\}$   from $[0, T]$ onto $H^{s}_\#$, such that
\begin{align}\label{E3.70}
u_m\to u \text{ in } W^1(0,T;H^s_\#,H^{s'}_\#)\quad  \text{as }m \rightarrow \infty.
\end{align}
For each $u_m$, we have
\begin{align} \label{E3.71}
\partial_t\|u_m(t)\|^2_{H^{(s+s')/2}_\#}
&=\partial_t\|\Lambda_\#^{(s+s')/2}u_m(t)\|^2_{H^{0}_\#}
=\partial_t\left\langle\Lambda_\#^{(s+s')/2}u_m(t), \Lambda_\#^{(s+s')/2}{u_m(t)}\right\rangle_{\T}\nonumber\\
&=2\Re\left\langle\Lambda_\#^{(s+s')/2}{u}_m^{\prime}(t), \Lambda_\#^{(s+s')/2}{u_m(t)}\right\rangle_{\T}
=2\Re\left\langle \Lambda_\#^{s'}u_m^{\prime}(t), \Lambda_\#^{s}{u_m(t)}\right\rangle_{\T}.
\end{align}
By \eqref{E3.70},
\begin{align*}
&\|u_m\|^2_{H^{s}_\#}=\|\Lambda_\#^{s}u_m\|^2_{L_{2\#}}\to \|\Lambda_\#^{s}u\|^2_{L_{2\#}}=\|u\|^2_{H^{s}_\#} \text { in } L_{1\#}(0, T), \\
&\|u'_m\|^2_{H^{s'}_\#}=\|\Lambda_\#^{s'}u'_m\|^2_{L_{2\#}}\to \|\Lambda_\#^{s'}u'\|^2_{L_{2\#}}=\|u'\|^2_{H^{s'}_\#} 
\text { in } L_{1\#}(0, T).
\end{align*}
Hence
\begin{align*} 
\left\langle \Lambda_\#^{s'}u_m^{\prime}, \Lambda_\#^{s}{u_m}\right\rangle_{\T} \to
\left\langle \Lambda_\#^{s'}u^{\prime}, \Lambda_\#^{s}{u}\right\rangle_{\T}\quad \text { in } {L}_{1\#}(0, T).
\end{align*}
These convergences also  hold for a.e. $t\in(0,T)$ and  in the distribution sense; therefore we are allowed to pass to the limit in \eqref{E3.71} in the distribution sense, arriving at \eqref{E3.69} in the limit.

(ii) Since $u\in W^1(0,T;H^s_\#,H^{s'}_\#)$ and $v\in W^1(0,T;H^{-s'}_\#,H^{-s}_\#)$, the dual products under the integral in \eqref{E3.30} are bounded in $L_1(0,T)$ and hence the integral is well defined. 
On the other hand, Theorem \ref{LM-T3.1} implies that $u$ and $v$ almost everywhere on $[0,T]$ equal to, respectively, functions   
$\tilde u\in \mathcal C^0([0,T];{H}_{\#}^{(s+s')/2})$ and $\tilde v\in \mathcal C^0([0,T];{H}_{\#}^{-(s+s')/2})$.
Then the traces $u(t), v(t), u(0),  v(0)$ are defined as the corresponding values of $\tilde u$ and $\tilde v$, implying that the dual products in the last two terms in \eqref{E3.30} are well defined.
Further in the proof we redefine $u$ and $v$ on a set of measure zero in $[0,T]$ as the functions $\tilde u$ and $\tilde v$, respectively.

There exists a sequence of infinitely differentiable functions $\{v_k\}$   from $[0, T]$ onto $H^{-s'}_\#$, such that
$v_k\to v \text{ in } W^1(0,T;H^{-s'}_\#,H^{-s}_\#),$  $k \rightarrow \infty$.
For each $u_m$ and $v_k$, we have
\begin{align*} 
&\left\langle u'_m(t), v_k(t)\right\rangle_{\T}+\left\langle u_m(t), v'_k(t)\right\rangle_{\T} = \partial_t\langle u_m(t),v_k(t)\rangle_\T,
\end{align*}
which after the integration in $t$ leads to
\begin{align*}
\int_0^t\left[\langle {u}'_m (\tau),{  v}_k (\tau)\rangle _{\T }
+\langle {  u}_m (\tau),{  v}'_k (\tau)\rangle _{\T }\right]d\tau
=\left\langle {  u}_m(t),{  v}_k(t)\right\rangle_{\T }
-\left\langle {  u}_m(0),{  v}_k(0)\right\rangle _{\T }.
\end{align*}
Taking the limits as $m\to\infty$ and $k\to \infty$, we get \eqref{E3.30}.
\end{proof}


\paragraph{Data availability statement}
This paper has no associated data.

\end{document}